\documentclass[a4paper,11pt,leqno]{amsart}
\usepackage{cite}
\usepackage[english]{babel}
\usepackage[utf8]{inputenc}
\usepackage[T1]{fontenc}
\usepackage{amsthm}
\usepackage{amssymb}
\usepackage{amsmath}

\usepackage{amsfonts}
\usepackage{graphicx}
\usepackage[all]{xy}
\newtheorem{Thm}{Theorem}[section]
\newtheorem{Prop}[Thm]{Proposition}
\newtheorem{Cor}[Thm]{Corollary}
\newtheorem{Apulause}[Thm]{Lemma}
\newtheorem{Quest}{Question}
\newtheorem{Def}[Thm]{Definition}
\theoremstyle{definition}

\newtheorem{Not}[Thm]{Remark}
\newtheorem{Ex}[Thm]{Example}
\renewcommand{\r}{|}
\newcommand{\R}{\mathbb{R}}

\renewcommand{\H}{\mathcal{H}}
\newcommand{\N}{\mathbb{N}}

\newcommand{\T}{\mathcal{T}}
\newcommand{\U}{\mathcal{U}}

\newcommand{\es}{\varnothing} % tyhjä joukko

\renewcommand{\Cup}{\bigcup}
\renewcommand{\Cap}{\bigcap}
\renewcommand{\unlhd}{\subset}
\newcommand{\ol}{\widetilde}
\newcommand{\ra}{\rightarrow}
\newcommand{\hra}{\hookrightarrow}
\newcommand{\law}{\leftarrow}
\newcommand{\al}{\alpha}
\newcommand{\be}{\beta}

\newcommand{\cu}{\curvearrowright}
\title{Monodromy representations of completed coverings}
\author{Martina Aaltonen}
\address{Department of Mathematics and Statistics, P.O.Box 68, 00014 Univeristy of Helsinki}
\thanks{The financial support from Emil Aaltonen Foundation is gratefully acknowledged.}
\email{martina.jansson@helsinki.fi}
\date{\today}  
 
\begin{document}

\begin{abstract}In this paper we consider completed coverings that are branched coverings in the sense of Fox. For completed coverings between PL manifolds we give a characterization of the existence of a monodromy representation and the existence of a locally compact mo\-no\-dromy representation. These results stem from a characterization for the discreteness of a completed normal covering. We also show that completed coverings admitting a monodromy representations are discrete and that the image of the branch set is closed.
\end{abstract}

\maketitle

\section{Introduction}

By the classical theory of covering spaces, a covering map $f  
\colon X\to Z$ between manifolds is a factor of a normal (or regular) covering; there exists a normal covering $p \colon Y \ra X$ so that $q = f \circ p \colon Y \to Z$ is a normal covering and the deck-transformation group of $q$ is isomorphic to the monodromy group of $f,$
\begin{equation*}\label{rr}
\xymatrix{
& Y \ar[ld]_p \ar[rd]^q &\\
X \ar[rr]^f & & Z. }
\end{equation*} 
In this case, the monodromy group $G$ of $f$ acts on $Y$ and there exists a subgroup $H \subset G$ for which $Y/G \approx Z$ and $Y/H \approx X.$ The normal coverings $p \colon Y \ra X$ and $q \colon Y \ra Z$ are orbit maps. In this article we are in\-te\-res\-ted in ramifications of this construction for discrete open mappings $f \colon X \ra Z$ between manifolds. 

By {\v{C}}ernavski ~\cite{CER} and Väisälä \cite{V} a discrete and open mapping $f \colon X \ra Z$ between manifolds is almost a local homeomorphism in the following sense. Let $B_f \subset X$ be the branch set i.e.\;the set of points in $X,$ where $f$ is not a local homeomorphism. By the theorem of {\v{C}}ernavski and Väisälä the topological codimension of $B_f$ is at least $2.$ In fact, if we set $Z':=Z \setminus f(B_f)$ and $X':=X \setminus f^{-1}(f(B_f)),$ then $X' \subset X$ is a dense and connected subset of $X$ and $Z' \subset Z$ is a dense and connected subset of $Z$ and $g:=f \r X' \colon X' \ra Z'$ is a local homeomorphism.

Suppose now in addition that $Z' \subset Z$ is open and $g$ is a covering. By the classical argument above there exists an open manifold $Y'$ and a commutative diagram of discrete and open mappings
\begin{equation}\label{hei}
\xymatrix{
& Y' \ar[ld]_{p'} \ar[rd]^{q'} &\\
X \ar[rr]^f & & Z }
\end{equation} 
where $p' \colon Y' \ra X'$ and $q' \colon Y' \ra Z'$ are normal coverings, the deck-trans\-for\-ma\-ti\-on group of the covering $q' \colon Y' \ra Z'$ is isomorphic to the mono\-dromy group of $g$ and $q'=f \circ p'.$ It becomes a question, whether there exists a space $Y \supset Y'$ so that $p'$ and $q'$ extend to discrete orbit maps $p \colon Y \ra X$ and $q \colon Y \ra Z$ satisfying $q=f \circ p.$ We formalize this question as follows.

\begin{Quest}\label{kysymys}Suppose $f \colon X \ra Z$ is an open discrete mapping between manifolds so that 
$Z \setminus f(B_f) \subset Z$ is open and 
$$f \r  \big(X \setminus f^{-1}(f(B_f))\big) \colon \big(X \setminus f^{-1}(f(B_f))\big) \ra \big(Z \setminus f(B_f)\big)$$ is a covering. Does there exists a locally connected Hausdorff space $Y,$ an embedding $\iota  \colon Y' \ra Y$ and orbit maps $p \colon Y \ra X$ and $q \colon Y \ra Z$ so that 
\begin{equation}\label{ww}
\xymatrix{
& Y \ar[ld]_p \ar[rd]^q &\\
X \ar[rr]^f & & Z}
\end{equation} 
is a commutative diagram of discrete and open mappings satisfying
$p'= p \circ \iota$ and $q'=q \circ \iota,$ where $p'$ and $q'$ are as in \eqref{hei}. 
Further, we require that $\iota(Y')\subset Y$ is a dense subset and $Y \setminus \iota(Y')$ does not locally separate $Y.$
\end{Quest}

This question stems from an article of Berstein and Edmonds \cite{BE} where they show that for open and discrete mappings between compact manifolds the answer to Question \ref{kysymys} is positive and the orbit maps $p$ and $q$ are induced by the action of the monodromy group of $f$ on $Y.$ They use this construction to give degree estimates for simplicial maps between compact manifolds; see also Pankka-Souto \cite{PS} for another application.

 In this article we extend the construction of Berstein and Edmonds \cite[2.2\;Prop.]{BE} for completed coverings (Definition \ref{cc}) that are branced coverings in the sense of Fox \cite{F}. Completed coverings form a subclass of open and surjective mappings from a Hausdorff space onto a manifold. 
Examples of completed coverings between manifolds are discrete and open mappings between compact manifolds (see \cite{CH}), proper quasiregular mappings (see \cite{R}) and surjective open and discrete simplicial mappings between PL manifolds (see \cite{I}). The class of completed coverings also include mappings outside these classes of mappings (see \cite{DP} and \cite{L}). The definition of a completed covering is technical and it comes from the theory of complete spreads in Fox \cite{F}. We only mention here, that the class of completed coverings is sufficiently large and natural in our setting.

Edmonds' results on orbit maps with finite multiplicity \cite{E} ensures that the answer to Question \ref{kysymys} is known for completed coverings that have finite multiplicity: By Edmods' result completed normal coverings that have finite multiplicity are discrete orbit maps; see \cite[Thm.\;4.1]{E}. Based on this the argument of Berstein and Edmonds \cite{BE} relies on the observation that a discrete and open mapping $f \colon X \ra Z$ between compact manifolds is a completed covering that has a finite monodromy group. For completed coverings the finiteness of the monodromy group is a sufficient condition for the argument of Bernestein and Edmonds. Thus, by the finiteness of the monodromy group, the answer to Question \ref{kysymys} is positive for every completed covering $f \colon X \ra Z$ between manifolds that has finite multiplicity.  

On the other hand, for Question \ref{kysymys} to have a positive answer for a mapping $f \colon X \ra Z$ having infinite multiplicity, it is not enough to assume that $f$ is a completed covering. The underlying reason is that the monodromy group of $f$ is in this case infinite, and a completed normal covering $p \colon Y \ra Z$ with infinite multiplicity is not necessary discrete; see Montesinos \cite[Ex.\;10.6]{MA}. To answer Question \ref{kysymys}, our first main Theorem is the following. 

\begin{Thm}\label{hilu}A completed normal covering $p \colon Y \ra Z$ from a Hausdorff space $Y$ onto a manifold $Z$ is an orbit map if and only if $p$ is a discrete map.
\end{Thm}

For this reason in the heart of Question \ref{kysymys} is the characterization of discrete completed normal coverings. In this direction we obtain the following results.
\begin{Thm}\label{ilta} A completed normal covering $p \colon Y \ra Z$ from a Hausdorff space $Y$ onto a PL manifold $Z$ is discrete if and only if $p$ is a stabily completed normal covering. 
\end{Thm}
Stabily completedness is defined in Section \ref{dcc}. By Montesinos a completed normal covering $p \colon Y \ra Z$ from a Hausdorff space $Y$ onto a manifold $Z$ is discrete if it has locally finite multiplicity; see  \cite[Thm. 9.14]{MA}. We prove the following.

\begin{Thm}\label{kuuskytkolme}Let $p \colon Y \ra Z$ be a discrete completed normal covering from a Hausdorff space $Y$ onto a PL manifold $Z.$ Then $Y$ is locally compact if and only if $p$ has locally finite multiplicity. 
\end{Thm}   

\noindent{\textbf{Regularity and existence of monodromy representations.}} We begin with a result showing the naturality of completed coverings. We show that if the space $Y$ in diagram \eqref{ww} is locally compact and locally connected, then all the maps in the diagram are completed coverings.  

\begin{Thm}\label{sade} Suppose $f \colon X \ra Z$ is a discrete and open mapping between ma\-nifolds,
$X':= X \setminus f^{-1}(f(B_f))$ and $Z':= Z \setminus f(B_f)$ for the branch set $B_f$ of $f.$ Suppose $Z'\subset Z$ is open and $f \r  X' \colon X' \ra Z'$ is a covering. Let 
\begin{equation*}
\xymatrix{
& Y' \ar[ld]_{p'} \ar[rd]^{q'} &\\
X \ar[rr]^f & & Z }
\end{equation*} 
be a commutative diagram of discrete and open mappings so that $p' \colon Y' \ra X'$ and $q' \colon Y' \ra Z'$ are normal coverings. Suppose $Y$ is a locally compact and locally connected Hausdorff space so that there exists an embedding $\iota \colon Y' \ra Y$ and orbit maps $p \colon Y \ra X$ and $q \colon Y \ra Z$ for which  
\begin{equation*}
\xymatrix{
& Y \ar[ld]_p \ar[rd]^q &\\
X \ar[rr]^f & & Z }
\end{equation*} 
is a commutative diagram of discrete and open mappings satisfying
$p'= p \circ \iota$ and $q'=q \circ \iota,$ and so that $\iota(Y')\subset Y$ is dense and $Y \setminus \iota(Y')$ does not locally separate $Y.$ 

Then $f \colon X \ra Z,$ $p \colon Y \ra X$ and $q \colon Y \ra Z$ are completed coverings.
\end{Thm}

Let $X$ and $Z$ be manifolds and $f \colon X \ra Z$ a completed covering. We say that a triple $(Y,p,q)$ is \textit{a monodromy representation} of $f$ if $Y$ is a locally connected Hausdorff space and the monodromy group $G$ of $f$ has an action on $Y$ and a subgroup $H \subset G$ so that $Y/G \approx Z,$ $Y/H \approx X$ and the associated orbit maps $p \colon Y \ra X$ and $q \colon Y \ra Z$ are completed coverings satisfying $q=p \circ f.$  
%\begin{equation*}
%\label{eq:kolmio}
%\xymatrix{
%& Y \ar[ld]_p \ar[rd]^q &\\
%X \ar[rr]^f & & Z. }
%\end{equation*}
We call a monodromy representation $(Y,p,q)$ \textit{locally compact} if $Y$ is locally compact. We present the following four regularity results for monodromy representations;

\begin{Thm}\label{ll} Let $f \colon X \ra Z$ be a completed covering between manifolds. Suppose $(Y,p,q)$ is a monodromy representation of $f.$ Then $f,$ $p$ and $q$ are discrete maps.
\end{Thm}

\begin{Thm}\label{yopo} Let $f \colon X \ra Z$ be a completed covering between manifolds. Suppose $(Y,p,q)$ is a monodromy representation of $f$ and $B_f \subset X,$ $B_p \subset Y$ and $B_q \subset Y$ the respective branch sets. Then $f(B_f) \subset Z,$ $p(B_p) \subset X$ and $q(B_q) \subset Z$ are closed sets. 
\end{Thm}

\begin{Thm}\label{mm} Let $f \colon X \ra Z$ be a completed covering between PL manifolds. Suppose $(Y,p,q)$ is a monodromy representation of $f.$ Then $f,$ $p$ and $q$ are stabily completed coverings. 
\end{Thm}

\begin{Thm}\label{kk} Let $f \colon X \ra Z$ be a completed covering between PL manifolds. Suppose $(Y,p,q)$ is a locally compact monodromy representation of $f.$ Then for every polyhedral path-metric $d_s$ on $Z$ there exists a path-metric $d_s^*$ on $Y,$ so that
\begin{itemize}
\item[(a)]the topology induced by $d_s^*$ coincides with the topology of $Y$,
\item[(b)]$(Y,d_s^*)$ is a locally proper metric space,
\item[(c)]$q \colon (Y,d_s^*) \ra (Z,d_s)$ is a $1$-Lipschitz map, and 
\item[(d)]the action of the monodromy group of $f$ in the monodromy representation $(Y,p,q)$ is by isometries on $(Y,d_s^*).$ 
\end{itemize}
\end{Thm}

We note that the path metric $d_s^*$ in Theorem \ref{kk} is a pullback of the path metric $d_s$ in $Z.$ Thus the map $q \colon (Y,d_s^*) \ra (Z,d_s)$ is a $1$-BLD map. Further, if $f \colon (X,e_s) \ra (Z,d_s)$ is a $L$-BLD map for a path metric $e_s$ on $X,$ then $p \colon (Y,d^*_s) \ra (X,e_s)$ is a $L$-BLD map. A mapping $f \colon X \ra Z$ between length manifolds is $L$-BLD for $L \geq 1$ if $f$ is discrete and open and satisfies 
$$\frac{1}{L} \ell(\gamma) \leq \ell(f \circ \gamma) \leq L \ell(\gamma)$$ for every path $\gamma$ in $X,$ where $\ell(\cdot)$ is the length of a path. We refer to \cite{HR} for a detailed discussion of BLD-maps. 

For the existence of a monodromy representations, we have the following characterization. Together with Theorem \ref{ll}, this answers to Question \ref{kysymys}  in the context of completed coverings between PL manifolds. 

\begin{Thm}\label{relas} A completed covering $f \colon X \ra Z$ between PL manifolds has a mo\-nodromy representation $(Y,p,q)$ if and only if $f$ is \textit{stabily completed}. 
\end{Thm}

As a corollary of Theorem \ref{relas} open and surjective simplicial mappings between PL manifolds have monodromy representations. Another corollary is the following.

\begin{Cor} Let $f \colon X \ra Z$ be a $L$-BLD mapping between PL $2$-ma\-ni\-folds. Then $f$ has a monodromy representation if and only if $f(B_f) \subset Z$ is a discrete set.
\end{Cor}
Indeed, if $f(B_f)\subset Z$ is a discrete set, then the BLD-mapping $f$ is a completed covering by Luisto \cite{L}, and hence $f$ is a stabily completed covering. Then $f$ has a monodromy representation by Theorem \ref{relas}. 

If $f(B_f) \subset Z$ is not a discrete set, then $f(B_f) \subset Z$ it not a closed set by Stoilow's Theorem. Thus, by Theorem \ref{yopo}, $f$ has no monodromy representation. 

We characterize the existence of locally compact monodromy representations as follows. Local monodromy groups are defined in Definition \ref{br}.

\begin{Thm}\label{hopo} A completed covering  $f \colon X \ra Z$ between PL manifolds has a locally compact monodromy representation $(Y,p,q)$ if and only if $f$ is stabily completed and $f$ has a finite local monodromy group at each point of $Z.$  
\end{Thm} 

We show that not every monodromy representation is locally compact. However, our results combined with previous results by Fox, Edmonds and Montesinos have an easy corollary in the positive direction. For the following statement we say that a covering $g \colon X' \ra Z'$ is \textit{virtually normal} if $g_*(\pi(X',x_0))\subset \pi(Z',z_0)$ contains a finite index subgroup which is normal in $\pi(Z',z_0).$ 

\begin{Cor}\label{ilo} Let $f \colon X \ra Z$ be a completed virtually normal covering between PL manifolds. Then $f$ has a locally compact monodromy representation $(Y,p,q),$ where $p$ has finite multiplicity. 
\end{Cor}
\begin{proof} By Fox \cite{F} we find a completed normal covering $p \colon Y \ra Z$ so that $p$ has finite multiplicity and $q:=f \circ p$ is a completed normal covering. Since $p$ has finite multiplicity, $p$ is an orbit map by Edmonds \cite[Thm.\;4.1]{E} and $q$ has locally finite multiplicity. Since $q$ has locally finite multiplicity, the map $q$ is a discrete map by Montesinos \cite[Thm. 9.14]{MA}. Thus $q$ is an orbit map by Theorem \ref{hilu} and $(Y,p,q)$ a monodromy representation of $f$. Since $q \colon Y \ra Z$ is a completed normal covering onto a PL manifold that is an orbit map and that has locally finite multiplicity, the space $Y$ is by Theorem \ref{kuuskytkolme} locally compact. Thus the monodromy representation $(Y,p,q)$ is locally compact.
\end{proof}

This article is organized as follows. In Sections \ref{nc} and \ref{mc} we discuss monodromy of covering maps. In Sections \ref{sp} and \ref{ccc} we introduce completed coverings. In Section \ref{os} we study the relation between  orbit maps and completed normal coverings and prove Theorem \ref{sade}. In Section \ref{ooo} we prove Theorem \ref{ll} by showing that competed normal coverings are orbit maps if and only if they are discrete maps. In Section \ref{ando} we prove Theorem \ref{yopo}. In Section \ref{dcc} we prove Theorems \ref{mm} and \ref{relas} by showing that a completed normal covering onto a PL manifold is discrete if and only if it is stabily completed. In Sections \ref{brr} and \ref{ccamm} we study completed normal coverings having locally finite multiplicity and prove Theorems \ref{kuuskytkolme}, \ref{kk} and \ref{hopo}. 

\tableofcontents

\noindent{\textbf{Acknowledgements.}} I want to thank my PhD advisor Pekka Pankka for our numerous discussions on the topic and for  reading the manuscript and giving excellent suggestions.

\section{Preliminaries}

\subsection{Normal coverings}\label{nc}
In this section we recall some facts on normal coverings and fix some notation related to coverings; we refer to Hatcher \cite[Ch.\;1]{H} for a detailed discussion.

Let $Y$ be a topological space. A \textit{path} is a continuous map $\al \colon [0,1] \ra Y.$ The \textit{inverse path} $\al^{-1} \colon [0,1] \ra Y$ of a path $\al \colon [0,1] \ra Y$ is the path $t \mapsto \al(1-t).$ The \textit{path composition} $\al\be \colon [0,1]\ra Y$ of paths $\al\colon [0,1] \ra Y$ and $\be\colon [0,1] \ra Y$ satisfying $\al(1)=\be(0)$ is the path $t \mapsto \al(2t)$ for $t \in [0,1/2]$ and $t \mapsto \be(2(t-1/2))$ for $t \in [1/2,1].$ For $y_0, y_1 \in Y$ a path $\al \colon [0,1] \ra Y$ satisfying $\al(0)=y_0$ and $\al(1)=y_1$ is denoted $\al \colon y_0 \cu y_1.$ Given $y_0 \in Y,$ we denote by $\pi(Y,y_0)$ the fundamental group of $Y$ (at $y_0$).

A map $p \colon X \ra Y$ between path-connected spaces $X$ and $Y$ is a covering map if for every $y \in Y$ there exists an open neighbourhood $U$ of $y$ such that the pre-image $p^{-1}(U) \subset X$ is a union of pairwise disjoint open sets homeomorphic to $U$. A space $X$ is called \textit{a cover} of $Y$ if there is a covering map from $X$ to $Y.$ 
Let $p \colon X \ra Y$ be a covering map, $y_0 \in Y$ and $x_0 \in p^{-1}\{y_0\}.$ A path $\ol{\al} \colon [0,1] \ra X$ is called a \textit{lift} of $\al \colon [0,1] \ra Y$ in $p$ if $\al= p \circ \ol{\al}.$ For every path $\al \colon [0,1] \ra Y$ satisfying $\al(0)=\al(1)=y_0$ there exists a unique lift $\ol{\al}$ satisfying $\ol{\al}(0)=x_0.$ We denote this lift by $\ol{\al}_{x_0}.$ Given a loop $\al \colon (S^1,e_0) \ra (Y,y_0)$ we denote by $\ol{\al}$ the lift of the path $[0,1] \ra Y, t \mapsto \al(\text{cos}(2 \pi t),\text{sin}(2 \pi t)).$

The \emph{deck-transformation group} $\T(p)$ of a map $p \colon X \ra Y$ is the group of all homeomorphisms $\tau \colon X \ra X$ satisfying $p(\tau(x))=p(x)$ for all $x \in X.$ Let $\tau_1 \colon X \ra X$ and $\tau_2 \colon X \ra X$ be deck-transformations of a covering $p \colon X \ra Y.$ Then $\tau_1=\tau_2$ if and only if there exists $x \in X$ so that $\tau_1(x)=\tau_2(x).$ 

A covering $p \colon X \ra Y$ is \textit{normal} if the subgroup $p_*(\pi(X,x_0)) \unlhd \pi(Y,y_0)$ is normal for $y_0 \in Y$ and $x_0 \in p^{-1}\{y_0\}.$ The following statements are equi\-valent: 
\begin{itemize}
\item[(a)]$p$ is a normal covering,
\item[(b)]for every $y \in Y$ and pair of points $x_1, x_2 \in p^{-1}\{y\}$ there exists a unique deck transformation $\tau \colon X \ra X$ satisfying $\tau(x_1)=x_2,$ 
\item[(c)]$X/\T(p) \approx Y,$ and
\item[(d)]$p_*(\pi(X,x_0))=p_*(\pi(X,x_1))$ for all $x_1 \in p^{-1}\{p(x_0)\}.$
\end{itemize}
Moreover, $\T(p) \cong \pi(Y,y_0)/p_*(\pi(X,x_0))$ if and only if $p$ is a normal covering.

For every manifold $Y,$ base point $y_0 \in Y$ and normal subgroup $N \subset \pi(Y,y_0),$ there exists a normal covering $p \colon X  \ra Y$ so that $p_*(\pi(X,x_0))=N$ for every $x_0 \in p^{-1}\{y_0\}.$ Moreover, the cover $X$ is a manifold. Thus $\T(p)$ is countable, since $p^{-1}\{y_0\}$ is countable. 

We also recall the following properties of covering maps. Let $p \colon (X,x_0) \ra (Y,y_0)$ be a normal covering and $\ol{\al}_{x_0} \colon [0,1] \ra X$ a lift of a loop $\al \colon (S^1,e_0) \ra (Y,y_0).$ Then the point $\ol{\al}_{x_0}(1) \in p^{-1}\{y_0\}$ depends only on the class of $\al$ in $\pi(Y,y_0)/p_*(\pi(X,x_0)).$ For any pair of points $y_1 \in Y$ and $x_1 \in p^{-1}\{y_1\}$ and a path $\be \colon y_1 \cu y_0,$ the lift $\ol{(\be\al\be^{\law})}_{x_1}$ is a loop if and only if the lift $\ol{\al}_{\ol{\be}_{x_1}(1)}$ is a loop. Further, $\ol{(\be\al\be^{\law})}_{x_1}(1)=\ol{(\be\gamma\be^{\law})}_{x_1}(1)$ for a loop $\gamma \colon (S^1,e_0) \ra (Y,y_0)$ if and only if $\ol{\al}_{\ol{\be}_{x_1}}(1)=\ol{\gamma}_{\ol{\be}_{x_1}}(1).$

\subsection{Monodromy}\label{mc}
In this section we recall some facts on the monodromy of maps. 

Let $f \colon X \ra Y$ be a covering, $y_0 \in Y$ and $K: = f^{-1}\{y_0\}.$ For every $[\al] \in \pi(Y,y_0)$ we define a map $m_{[\al]} \colon K \ra K$ by setting $m_{[\al]}(k)=\ol{\al}_k(1)$ for every $k \in K.$ For every $[\al] \in \pi(Y,y_0),$ the map $m_{[\al]}$ is a bijection, since 
$m_{[\al]} \circ m_{[\al]^{-1}}=m_{[\al]^{-1}} \circ m_{[\al]}=\mathrm{id}.$ The \textit{monodromy} of $f$ is the homomorphism 
$$\sigma_f \colon \pi(Y,y_0) \ra \text{Sym}(K),\, [\al] \to m_{[\al]}.$$ 
We call the quotient $\pi(Y,y_0)/\mathrm{Ker}(\sigma_f)$ the monodromy group of $f.$ 
%We note that $\pi(Y,y_0)/\mathrm{Ker}(\sigma_f) \cong \mathrm{Im}(\sigma_f)$ and we may refer to  the group $\mathrm{Im}(\sigma_f)$ as the monodromy group of $f.$  

Let $p \colon (Y,y_0) \ra (X,x_0)$ be a normal covering, $f \colon (X,x_0) \ra (Z,z_0)$ a covering and $q:=f \circ p \colon (Y,y_0) \ra (Z,z_0).$ Then $q$ is a covering and $q_{*}(\pi(Y,y_0))$ is a normal subgroup of $f_{*}(\pi(Z,z_0)).$ We note that, $q$ is a normal covering if and only if $q_{*}(\pi(Y,y_0))$ is a normal subgroup of $\pi(Z,z_0).$ The question weather the covering $q$ is a normal is related to the monodromy of the covering $f$ in the following way.

\begin{Prop}\label{spread_3}
Let $X$ be a manifold, $f \colon (X,x_0) \ra (Z,z_0)$ a covering and $N \subset \pi(Z,z_0)$ a normal subgroup. Then there exists a normal covering $p \colon (Y,y_0) \ra (X,x_0)$ satisfying
$(f \circ p)_{*}(\pi(Y,y_0))=N$ if and only if $N \subset \mathrm{Ker}(\sigma_f).$  
\end{Prop} 

In order to prove Proposition \ref{spread_3}, we prove first Lemma \ref{nalkat} and Lemma \ref{nalka}.

\begin{Apulause}\label{nalkat}Let $f \colon (X,x_0) \ra (Y,y_0)$ be a covering map. Then $\mathrm{Ker} (\sigma_f)$ is a normal subgroup of $\pi(Y,y_0)$ satisfying $\mathrm{Ker} (\sigma_f) \subset f_*(\pi(X,x_0)).$ 
\end{Apulause}
\begin{proof}Clearly, $\mathrm{Ker} (\sigma_f) \subset \pi(Y,y_0)$ is a normal subgroup. Let $\al \colon (S^1,e_0) \ra (Y,y_0)$ be a loop such that $[\al] \in \mathrm{Ker} (\sigma_f)$ and let $\ol{\al}_{x_0} \colon [0,1] \ra X$ be the lift of $\al$ in $f$ starting from $x_0.$  Since
$$\ol{\al}_{x_0}(1)=m_{[\al]}(x_0)=\text{id}(x_0)=x_0=\ol{\al}_{x_0}(0),$$ $\ol{\al}_{x_0}$ is a loop. Since $\al=f \circ \ol{\al}_{x_0},$ we have $[\al] \in f_*(\pi(X,x_0)).$ Hence $\mathrm{Ker} (\sigma_f) \subset f_*(\pi(X,x_0)).$ 
\end{proof}

\begin{Apulause}\label{nalka}Let $f \colon (X,x_0) \ra (Y,y_0)$ be a covering and $N \subset f_*(\pi(X,x_0))$ a normal subgroup of $\pi(Y,y_0).$ Then $N \subset \mathrm{Ker} (\sigma_f).$
\end{Apulause}

\begin{proof} We need to show that for every loop $\al \colon (S^1,e_0) \ra (Y,y_0)$ for which $[\al] \in N$ the lift $\ol{\al}_x$ in $f$ is a loop for every $x \in f^{-1}\{y_0\}$. Let $\be \colon  x_0 \cu x$ be a path. Since $N$ is a normal subgroup of $\pi(Y,y_0),$ we have
$$[\gamma]:=[f \circ \be][\al][f \circ \be]^{-1}=[(f \circ \be)\al(f \circ \be)^{\law}] \in N\subset f_*(\pi(X,x_0)).$$ 
Thus the lift $\ol{\gamma}_{x_0}$ is a loop. By the uniqueness of lifts, $\ol{\al}_x$ is a loop. This shows that $N \subset \mathrm{Ker}(\sigma_f).$  
\end{proof}

\begin{proof}[Proof of Proposition \ref{spread_3}] Let $p \colon (Y,y_0) \ra (Z,z_0)$ be a normal covering. Suppose first that $(f \circ p)_{*}(\pi(Y,y_0))=N.$ Then $N \subset \mathrm{Ker}(\sigma_f)$ by Lemma \ref{nalka}.

Suppose now that $N \subset \mathrm{Ker}(\sigma_f).$ Then, by Lemma \ref{nalkat}, $N \subset f_*(\pi(X,x_0)).$ Since $f_*$ is a monomorphism, the pre-image $\tilde{N}$ of $N$ under $f_*$ is a normal subgroup of $\pi(X,x_0)$ and isomorphic to $N.$ Hence there exists a normal covering $p \colon (Y,y_0) \ra (X,x_0)$ satisfying $p_{*}(\pi(Y,y_0))=\ol{N} \cong N.$ In particular, $(f \circ p)_{*}(\pi(Y,y_0))=f_{*}(\tilde{N})=N.$
\end{proof}

The following observation is due to Berstein-Edmonds \cite{BE}.

\begin{Prop}\label{he} Let
\begin{displaymath}
\xymatrix{&Y \ar[dl]_{p} \ar[dr]^{q:=f \circ p} \\X \ar[rr]_{f} && Z}
\end{displaymath}
be a diagram of covering maps so that $q_{*}(\pi(Y,y_0))=\mathrm{Ker}(\sigma_f)$ for $z_0 \in Z$ and $y_0 \in q^{-1}\{z_0\}.$ Then the deck-transformation group $\T(q)$ of the covering $q \colon Y \ra Z$ is finite if and only if the set $f^{-1}\{z_0\} \subset X$ is finite.  
\end{Prop}
\begin{proof} Suppose that the group $\T(q)$ is finite and $\#\T(q) = m$. Since $q$ is a normal covering, $\#(q^{-1}\{z_0\})=m$. Hence $\#(f^{-1}\{z_0\})\leq m,$ since $q=f \circ p.$
 
Suppose that $f^{-1}\{z_0\} \subset X$ is finite. Then the finiteness of the symmetric group $\text{Sym}(f^{-1}\{z_0\})$ implies that $$\T(q) \cong \pi(Z,z_0)/\mathrm{Ker}(\sigma_f) \cong \mathrm{Im}(\sigma_f) \subset \text{Sym}(f^{-1}\{y_0\})$$ is a finite group. 
\end{proof}

Next we define monodromy for a broader class of maps including completed coverings defined in the following section. We say, that an open dense subset $Z' \subset Z$ is \emph{large} if $Z \setminus Z'$ does not locally separate $Z.$ 

\begin{Apulause}\label{something}
Let $f \colon X \ra Z$ be a continuous and open map. Suppose that for $i \in \{1,2\}$ there exist large subsets $X_i \subset X$ and $Z_i \subset Z$ so that  
$g_i:=f \r X_i \colon X_i \ra Z_i$ is a covering. Then $\mathrm{Im}(\sigma_{g_1})$ is isomorphic to $\mathrm{Im}(\sigma_{g_2}).$ 
\end{Apulause}
\begin{proof} Let $z_0 \in Z_1 \cap Z_2$ and $g:= f \r X_1 \cap X_2 \colon X_1 \cap X_2 \ra Z_1 \cap Z_2.$ Then $g$ is a covering. It is sufficient to show that 
$\mathrm{Im}(\sigma_{g_1})=\mathrm{Im}(\sigma_{g}).$
Let $\iota_* \colon \pi(Z_1 \cap Z_2,z_0) \ra \pi(Z_1,z_0)$ be the homomorphism induced by the inclusion $\iota \colon Z_1 \cap Z_2 \hra Z_1.$ Then $\iota_*$ is surjective, since $Z_1 \cap Z_2 \subset Z_1$ is dense and $Z_1 \setminus(Z_1 \cap Z_2)$ does not locally separate $Z_1.$ Thus $\mathrm{Im}(\sigma_{g_1})\cong\mathrm{Im}(\sigma_{g}).$
\end{proof}

Let $f \colon X \ra Z$ be an open and continuous map so that $g:=f \r X' \colon X' \ra Z'$ is a covering for large subsets $X' \subset X$ and $Z' \subset Z.$ Then we say that the monodromy $\sigma_g \colon \pi(Z',z_0) \ra \mathrm{Sym}(f^{-1}\{z_0\})$ is a \textit{monodromy} $\sigma_f$ of $f.$ We call the monodromy group $\pi(Z',z_0)/\mathrm{Ker}(\sigma_g)$ of $g$ the \textit{monodromy group of} $f.$ 

\begin{Not}We note that the isomorphism class of the monodromy group of $f$ is independent of the choice of $z_0 \in Z',$ and thus by Lemma \ref{something}, it is independent of the choice of monodromy. 
\end{Not} %We note that the isomorphism class of $\pi(Z,z_0)/\mathrm{Ker}(\sigma_g)$ does not depend of the choice of the base point for $Z'$ or the choice of the monodromy $\sigma_f.$

\section{Spreads and completed coverings}

\subsection{Spreads}\label{sp}
In this section we discuss the theory of spreads; see Fox \cite{F} and Montesinos \cite{MA} for a more detailed discussion. 

\begin{Def}A mapping $g \colon Y \ra Z$ between locally connected Hausdorff spaces is called a \textit{spread} if the components of the inverse images of the open sets of $Z$ form a basis of $Y.$ 
\end{Def}

Clearly spreads are continuous and a composition of spreads is a spread. A map $g \colon Y \ra Z$ between topological manifolds is a spread if and only if $g$ is light; for every $z \in Z$ the set $g^{-1}\{z\}$ is totally disconnected, see \cite[Cor.4.7]{MA}. In particular, if $Z$ is a manifold and $g \colon Y \ra Z$ is an open mapping which is a covering to its image, then $g$ is a spread.  

The definition of a complete spread is formulated by Fox \cite{F}: a spread $g \colon Y \ra Z$ is \textit{complete} if for every point $z$ of $Z$ and every open neighbourhood $W$ of $z$ there is selected a component $V_W$ of $g^{-1}(W)$ in such a way that $V_{W_1} \subset V_{W_2}$ whenever $W_1 \subset W_2,$ then $\cap_{W}V_{W}\neq \es.$ This definition can equivalently be written as follows.

For every $z \in Z,$ we denote by $\mathcal{N}_Z(z)$ the set of open connected neighbourhoods of $z.$ A function $\Theta \colon \mathcal{N}_Z(z) \ra \text{Top}(Y)$ is a \emph{selection function} of a spread $g\colon Y \ra Z$ if $\Theta(W)$ is a component of $g^{-1}(W)$ for every $W \in \mathcal{N}_Z(z)$ and $\Theta(W_1)\subset \Theta(W_2)$ whenever $W_1 \subset W_2.$

\begin{Def} A spread $g \colon Y \ra Z$ is \textit{complete} if for all $z \in Z$ every selection function $\Theta \colon \mathcal{N}_Z(z) \ra \mathrm{Top}(Y)$ of $g$ satisfies $$\bigcap_{W \in \mathcal{N}_Z(z)} \Theta(W)\neq \es.$$  
\end{Def}

\begin{Not}Let $g \colon Y \ra Z$ be a complete spread. Since $Y$ is Hausdorff, the set $\bigcap_{W \in \mathcal{N}_Z(z)} \Theta(W)\subset Y$ is a point for every $z \in Z$ and every selection function $\Theta \colon \mathcal{N}_Z(z) \ra \text{Top}(Y).$ Moreover, for every $z \in Z$ and $y \in g^{-1}\{z\}$ there exists a selection function $\Theta \colon \mathcal{N}_Z(z) \ra \text{Top}(Y)$
of $g$ for which $\bigcap_{W \in \mathcal{N}_Z(z)} \Theta(W)=\{y\}.$
\end{Not}

\begin{Apulause}\label{aurinko}Let $g \colon Y \ra Z$ be a spread. Suppose that for every $z \in Z$ the following conditions are satisfied: 
\begin{itemize}
\item[(1)]$g(V)=W$ for every $W \in \mathcal{N}_Z(z)$ and every component $V$ of $g^{-1}(W),$ and
\item[(2)]there exists $W_0 \in \mathcal{N}_Z(z)$ so that for every component $V$ of $g^{-1}(W_0)$ the set $V \cap g^{-1}\{y\}$ is finite.
\end{itemize}
Then $g$ is a complete spread.
\end{Apulause}
\begin{proof} Let $z \in Z$ and let $\Theta \colon \mathcal{N}_Z(z) \ra \text{Top}(Y)$ be a selection function for $g$. We need to show that $\bigcap_{W \in \mathcal{N}_Z(z)} \Theta(W) \neq \es.$ 

Suppose $\bigcap_{W \in \mathcal{N}_Z(z)} \Theta(W)=\es.$ By condition (2), the set $g^{-1}\{z\} \cap \Theta(W_0)$ is finite for some $W_0 \in \mathcal{N}_Z(z).$ Let $\{y_1, \ldots, y_k\}=g^{-1}\{z\} \cap \Theta(W_0).$ Since $\bigcap_{W \in \mathcal{N}_Z(z)} \Theta(W)=\es,$ we find for every $j=0, \ldots, k$ a set $W_j \in \mathcal{N}_Z(z)$ for which $y_j \notin \Theta(W_j).$ Let $W'=W_0 \cap \cdots \cap W_k.$ Then $\Theta(W') \cap g^{-1}\{z\}=\es.$ Thus $z \notin \bigcap_{W \in \mathcal{N}_Z(z)}g(\Theta(W))\subset g(\Theta(W')).$

On the other hand, $z \in \bigcap_{W \in \mathcal{N}_Z(z)}g(\Theta(W))$ as a consequense of condition (1). This is a contradiction and we conclude that $\bigcap_{W \in \mathcal{N}_Z(z)} \Theta(W)\neq \es.$
\end{proof}

For completion of presentation we include the following proposition well known for the experts.

\begin{Prop} Let $f \colon X \ra Z$ be a discrete and open map between compact manifolds. Then $f$ is a complete spread.
\end{Prop}
\begin{proof} The map $f$ is a spread by \cite[Cor.4.7]{MA}; see also proof of Theorem \ref{pakkanen}. We prove that $f$ is complete by showing that conditions (1) and (2) in Lemma \ref{aurinko} hold. Since $f$ is an open and discrete map from a compact manifold $X,$ the map $f$ has finite multiplicity. Thus condition (2) holds since $f$ is a spread. We then suppose that condition (1) does not hold. Then there exists $x \in X$ and $V \in \mathcal{N}_Z(f(x))$ so that the $y$-component $U$ of $f^{-1}(V)$ satisfies $f(U)\subsetneq V.$ Since $V$ is connected there is a point $z \in \partial f(U) \cap V.$ Since $X$ is compact $\overline{U} \subset X$ is compact. Thus $z \in \overline{f(U)} \subset f(\overline{U})$ and  there exists a point $x' \in \overline{U} \cap p^{-1}\{z\}.$ This is a contradiction since $x' \in (\overline{U} \setminus U) \cap p^{-1}(V)$ and $U$ is a component of $p^{-1}(V).$ We conclude that (1) holds and that $f$ is a complete spread.
\end{proof}

We recall that an open dense subset $Z' \subset Z$ is large if $Z \setminus Z'$ does not locally separate $Z.$

\begin{Def}\label{completion}Let $X,$ $Y$ and $Z$ be Hausdorff-spaces. A complete spread $g \colon Y \ra Z$ is a \textit{completion} of a spread $f \colon X \ra Z$ if there is an embedding $i \colon X \ra Y$ such that $i(X)\subset Y$ is large and 
\begin{displaymath} 
\xymatrix{X \ar[rr]^{i} \ar[dr]_{f} && Y \ar[dl]^g  \\&Z.}
\end{displaymath}
\end{Def}

Two completions $g_1 \colon Y_1 \ra Z$ and $g_2 \colon Y_2 \ra Z$ of a spread $f \colon X \ra Z$ are \textit{equivalent} if there exists a homeomorphism $j \colon Y_1 \ra Y_2$ such that $g_2 \circ j=g_1$ and $j(x)=x$ for all $x \in X.$ By the fundamental result of Fox, every spread has a completion \cite[Sec.\;2]{F} and it is unique up to equi\-va\-len\-ce \cite[Sec.\;3]{F} ; see also \cite[Cor. 2.8 and Cor. 7.4]{MA}. 

\subsection{Completed coverings}\label{ccc}

Let $Y$ be a Hausdorff space and $Z$ a manifold. Suppose $f \colon Y \ra Z$ is a complete spread so that there are large subsets $Y' \subset Y$ and $Z' \subset Z$ for which $f':=f \r Y' \colon Y' \ra Z'$ is a covering. Then $f$ is the completion of $f' \colon Y' \ra Z,$ 
and since $Z' \subset Z$ is large, $f$ is an open surjection, see \cite[Teo. 2.1]{C}.

\begin{Def}\label{cc}
Let $Y$ be a Hausdorff space and $Z$ a manifold. A map $f \colon Y \ra Z$ is a \textit{completed covering} if there are large subsets $Y' \subset Y$ and $Z' \subset Z$ so that $f\r Y' \colon Y' \ra Z'$ is a covering and $f$ is the completion of $f'\r Y' \colon Y' \ra Z.$ 
\end{Def}

The results in \cite{I} cover the following example.

\begin{Ex}Let $f \colon X \ra Z$ be an open and discrete simplicial map from a PL $n$-manifold $X$ with triangulation $K$ onto a PL $n$-manifold $Z$ with triangulation $L.$ Let $K^{n-2}$ be the $(n-2)$-skeleton of $K,$ $L^{n-2}$ the $(n-2)$-skeleton of $L,$ $X'=|K| \setminus |K^{n-2}|$ and $Z'=|L| \setminus |L^{n-2}|.$ Then $f \r  X' \colon X' \ra Z'$ is a covering and $f$ is the completion of $f \r X' \colon X' \ra Z.$  
\end{Ex}

Next we prove some basic properties of completed coverings.

\begin{Apulause}\label{biotin}Let $g \colon Y' \to Z'$ be a normal covering onto a large subset $Z'\subset Z$ of a manifold $Z$ and $f \colon Y \ra Z$ the completion of $g \colon Y' \ra Z.$ Then $f^{-1}(Z')=Y'.$
\end{Apulause}
\begin{proof} Suppose there exists a point $y \in Y \setminus Y'$ so that $z:=f(y)\in Z'.$ Let $\Theta \colon \mathcal{N}_Z(z) \ra \text{Top}(Y)$ be a selection function of $f$ satisfying 
$$\bigcap_{W \in \mathcal{N}_Z(z)} \Theta(W)=\{y\}.$$
Since $Z' \subset Z$ is open, we have $\mathcal{N}_{Z'}(z) \subset \mathcal{N}_Z(z).$  Moreover, since $Y' \subset Y$ is large, the set $\Theta(W) \cap Y'$ is a component of $g^{-1}(W)$ for every $W \in \mathcal{N}_{Z'}(z).$ Define $\Theta' \colon \mathcal{N}_{Z'}(z) \ra \text{Top}(Y')$ by $\Theta'(W)= \Theta(W) \cap Y'.$ Then $\Theta'$ is a selection function for $g \colon Y' \ra Z'.$
Since $g$ is a covering, it is also a complete spread. Thus there exists $y' \in Y'$ so that 
$$\bigcap_{W \in \mathcal{N}_{Z'}(z)} \Theta'(W)=\{y'\}.$$ This is a contradiction, since
$$y' \in \bigcap_{W \in \mathcal{N}_Z(z)} \Theta(W)=\{y\} \subset Y \setminus Y'.$$
\end{proof}

\begin{Apulause}\label{brunate} Let $g \colon Y' \to Z'$ be a covering onto a large subset $Z'\subset Z$ of a manifold $Z$ and $f \colon Y \ra Z$ the completion of $g \colon Y' \ra Z.$ Let $U \subset Z$ be a connected and open set, and let $V$ be a component of $f^{-1}(U).$ Then $f(V)=U.$ 
\end{Apulause}
\begin{proof} Since $Z'\subset Z$ is large, the set $U \cap Z'$ is path-connected. Furthermore, since $Y' \subset Y$ is dense, we may fix a component $W$ of $g^{-1}(U \cap Z')$ contained in $V \cap Y'.$ Every path into $U \cap Z'$ has a total lift into $W,$ since $g$ is a covering map. Thus $U \cap Z'= \mathrm{Im}(g \r W).$ In particular, $g \r W \colon W \ra Z' \cap U$ is a covering onto a large subset of $U$. 

Since $W$ is a connected and open set and $f$ is a complete spread, there exists such a connected subset $V' \subset Y$ that $f \r V'$ is the completion of the spread $g \r W \colon W \ra U.$ Since $U \cap Z'\subset U$ is large, the map $f \r V' \colon V' \ra U$ is a surjective spread by \cite[Thm. 2.1]{C}. Since $V'$ is connected and $W \subset V,$ we have $V' \subset V.$ Thus $U \subset f(V') \subset f(V)\subset U.$
\end{proof}

\begin{Apulause}\label{elama} Let $f_1 \colon X \ra Y$ and $f_2 \colon Y \ra Z$ be completed coverings. Then $f_2 \circ f_1 \colon X \ra Z$ is a completed covering.
\end{Apulause}
 
\begin{proof} Let $X'\subset X,$ $Y' \subset Y$ and $Z' \subset Z$ be large subsets so that $g_1:= f_1 \r X' \colon X' \ra Y'$ and $g_2:= f_1 \r Y' \colon Y' \ra Z'$ are coverings. Then $f_1$ is the completion of $g_1 \colon X' \ra Y$ and $f_2$ is the completion of $g_2 \colon Y' \ra Z$ since $f_1$ and $f_2$ are complete spreads. Since $Z'\subset Z$ is large $f_2\circ f_1 \colon X' \ra Z$ is a spread. Moreover, since $X' \subset X$ is large, it is sufficient to show that $f_2 \circ f_1 \colon X \ra Z$ is complete. 

If there are no selection functions for $f_2 \circ f_1,$ then $f_2 \circ f_1$ is trivially a complete spread. Then suppose $z \in Z$ and $\Theta \colon \mathcal{N}_{Z}(z) \ra \text{Top}(X)$ is a selection function for $f_2 \circ f_1.$ We need to show that $\Cap_{W \in \mathcal{N}_Z(z)} \Theta(W) \neq \es.$ 

Since $\Theta$ is a selection function, there are for every $W \in \mathcal{N}_Z(z)$ components $U_{W}$ of $f_2^{-1}(W)$ and $V_{U_{W}}$ of $f_2^{-1}(U_{W})$ so that
$\Theta(W)=V_{U_{W}}$ and $U_{W_1} \subset U_{W_2}$ whenever $W_1 \subset W_2.$ Define $\Theta_2 \colon \mathcal{N}_{Z}(z) \ra \text{Top}(Y)$ by $W \mapsto U_{W}.$ Then $\Theta_2$ is a selection function of $f_2.$ %Indeed, the condition that $\Theta_2(W_1)\subset \Theta_2(W_2)$ whenever $W_1 \subset W_2$ is satisfied by selection of $U_W.$ 

Since $f_2$ is a complete spread, there exists a point $y \in f^{-1}_2\{z\}$ satisfying $\Cap_{W \in \mathcal{N}_Z(z)} \Theta_2(W)=\{y\}.$ Thus there exists a selection function $\Theta_1 \colon \mathcal{N}_{Y}(y) \ra \text{Top}(X)$ satisfying $\Theta_1(U_W)=V_{U_{W}}$ for every
$W \in \mathcal{N}_Z(z).$ Since $f_1$ is complete, there exists a point $x \in f_1^{-1}\{y\}$ satisfying $\Cap_{U \in \mathcal{N}_Y(y)}\Theta_1(U)=\{x\}.$
Now
$$\{x\}=\Cap_{U \in \mathcal{N}_Y(y)} \Theta_1(U) \subseteq \Cap_{W \in \mathcal{N}_Z(z)} V_{U_W}= \Cap_{W \in \mathcal{N}_Z(z)} \Theta(W).$$
Thus $ \Cap_{W \in \mathcal{N}_Z(z)} \Theta(W) \neq \es,$ and we conclude that $f_2\circ f_1$ is a complete spread.
\end{proof}

\begin{Apulause}\label{sina} Let $f_1 \colon X \ra Y$ be a completed covering and $f_2 \colon Y \ra Z$ a spread between manifolds and suppose $Y'' \subset Y$ and $Z'' \subset Z$ are large subsets so that the map $f_2 \r Y'' \colon Y'' \ra Z''$ is a covering. Then $f_2$ is a completed covering if $f:=f_2 \circ f_1 \colon X \ra Z$ is a completed covering.
\end{Apulause}
\begin{proof}Let $X'\subset X,$ $Y' \subset Y$ and $Z' \subset Z$ be large subsets so that the maps $g_1:= f_1 \r X' \colon X' \ra Y'$ and $f_2 \r Y' \colon Y' \ra Z'$ are coverings. Then $f_1$ is the completion of $g_1 \colon X' \ra Y$, $g:= f \r X' \colon X' \ra Z'$ is a covering and $f$ is the completion of $g \colon X' \ra Z$. Since $f_2$ is a spread and $Y' \subset Y$ is large, it is sufficient to show that $f_2 \colon Y \ra Z$ is complete.
 
Suppose there are no selection functions for $f_2.$ Then $f_2$ is trivially complete. Suppose then that $z \in Z$ and $\Theta \colon \mathcal{N}_{Z}(z) \ra \text{Top}(Y)$ is a selection function for $f_2.$ Since $f_1$ is a completed covering, there exists a selection function $\Theta' \colon \mathcal{N}_{Z}(z) \ra \text{Top}(X)$ so that $\Theta'(W)$ is a component of $f_1^{-1}(\Theta(W))$ for every $W \in \mathcal{N}_Z(z)$ as a consequence of Lemma \ref{brunate}. Since $f$ is complete, there exists a point $x \in \Cap_{W \in \mathcal{N}_Z(z)} \Theta'(W).$ 
Now 
$$f_1(x) \in f\left(\Cap_{W \in \mathcal{N}_Z(z)} \Theta'(W)\right) \subseteq\Cap_{W \in \mathcal{N}_Z(z)} f(\Theta'(W))\subseteq\Cap_{W \in \mathcal{N}_Z(z)} \Theta(W).$$
Thus $f_2$ is complete and we conclude that $f_2$ is a completed covering.
\end{proof}

\subsection{Orbit maps}\label{os} In this section we prove Theorem \ref{sade} in the introduction.

An open continuous map $p \colon Y \ra Z$ is an \textit{orbit map} if there exists a subgroup $G \subset \text{Homeo}(Y)$ and a homeomorphism $\phi \colon Y/G \ra Z$ so that the diagram 
\begin{displaymath} 
\xymatrix{Y \ar[rr]^{p} \ar[dr]_{\pi} && Z  \\& Y/G \ar[ru]_{\phi}},
\end{displaymath}
commutes for the canonical map $\pi \colon Y \ra Y/G$ 

\begin{Apulause}\label{vindpark}Let $f \colon Y \ra Z$ be an orbit map. Let $U \subset Z$ be a connected and open set and let $V_1$ and $V_2$ be components of $f^{-1}(U)$ satisfying $f(V_1)=f(V_2)=U.$ Then there exists a deck-transformation $\tau \colon Y \ra Y$ satisfying $\tau(V_1)=V_2.$
\end{Apulause}
\begin{proof} Let $z \in U$ and fix points $y_1 \in V_1 \cap f^{-1}\{z\}$ and $y_2 \in V_2 \cap f^{-1}\{z\}.$ Since $f$ is an orbit map there is a deck-transformation $\tau \colon Y \ra Y$ for which $\tau(y_1)=y_2.$ Since $\tau(f^{-1}(U))=f^{-1}(U)$ and $V_1$ and $V_2$ are components of $f^{-1}(U)$, we have $\tau(V_1)\subset V_2$ and $\tau^{-1}(V_2)\subset V_1.$ Thus $\tau(V_1)=V_2.$   
\end{proof}

\begin{Thm}\label{pakkanen} Let $Y$ be a locally compact and locally connected Hausdorff space, $Z$ a manifold and $p \colon Y \ra Z$ a discrete, continuous and open orbit map. Suppose there are large subsets $Y' \subset Y$ and $Z' \subset Z$ so that $g \colon Y' \ra Z'$ is a covering. Then $p$ is a completed covering. 
\end{Thm}
\begin{proof}
Suppose $U\subset Y$ is open. To show that $p$ is a spread, it is sufficient to show that for every $y \in U$ there exists a neighbourhood $V \subset p(U)$ of $p(y)$ for which the $y$-component $D$ of $p^{-1}(V)$ is contained in $U.$ Fix $y \in U.$ Since $Y$ is a locally compact and $p$ discrete there exists a neighbourhood $W\subset U$ of $y,$ so that $\overline{W}$ is compact and $\partial W \cap p^{-1}\{p(y)\}=\es.$ Since $p$ is open and continuous the set $p(W) \setminus p(\partial W)\subset Z$ is open. Since $Z$ is locally connected the $p(y)$-component $V$ of $p(W) \setminus p(\partial W)$ is open and since $Y$ is locally connected the $y$-component $D$ of $f^{-1}(V)$ is open. Since $D \cap \partial W=\es,$ we get $D \subset W \subset U.$  We conclude that $p$ is a spread.

Since $Y' \subset Y$ and $Z' \subset Z$ are large and $p \r Y' \colon Y' \ra Z'$ is a covering, to show that $p$ is a completed covering it is sufficient to show that the spread $p$ is complete.

For this let $z \in Z.$ By Lemma \ref{aurinko} it is suffient to show that $p(U)=W$ for every $W \in \mathcal{N}_Z(z)$ and $U$ a component of $p^{-1}(W),$ and that there is a connected neighbourhood $W_0$ of $z$ so that for every component $V_0$ of $p^{-1}(W_0)$ the set $V_0 \cap p^{-1}\{z\}$ is finite.
Towards this let $W \in \mathcal{N}_Z(z).$ Since $p$ is an orbit map $p$ is necessarily onto. Hence there is a component $U_1$ of $p^{-1}(W)$ for which $z \in p(U_1).$ Let $U_2$ be a component of $p^{-1}(W).$ Since $p$ is an orbit map the deck-transformation group $\T(p)$ acts transitively on fibers $p^{-1}\{z'\}$ for every $z' \in Z$. Thus Lemma \ref{vindpark} implies that $p(U_1)=p(U_2)$ or $p(U_1) \cap p(U_2)=\es.$ This implies that both $p(U_1)$ and $W \setminus p(U_1)$ are open sets. Since $W$ is connected, this is only possible if $W \setminus p(U_1)=\es.$ Thus $p(U)=W$ for every component $U$ of $p^{-1}(W).$   

Let then $y \in p^{-1}\{z\}$. Since $Y$ is locally compact and locally connected there is a neighbourhood $U_0$ of $y$ so that $p^{-1}\{z\}\cap U_0$ is finite and $U_0$ is the $y$-component of a $p^{-1}(p(U_0)).$ Further by Lemma \ref{vindpark}, $\#(p^{-1}\{z\}\cap V_0)=\#(p^{-1}\{z\}\cap U_0)$ for every component $V_0$ of $p^{-1}(p(U_0))$. Thus $p$ is a complete spread by Lemma \ref{aurinko}.
\end{proof}

\begin{proof}[Proof of Theorem \ref{sade}] Let $f \colon X \ra Z$ be a mapping and $X' \subset X$ and $Z' \subset Z$ large subsets as in the statement. Similarly, let $Y$ be a space and $p \colon Y \ra X$ and $q \colon Y \ra Z$ be orbit maps as in the statement. By assumption, there exists a large subset $Y' \subset Y$ for which the maps $p \r Y' \colon Y' \ra X'$ and $q \r Y' \colon Y' \ra X'$ are normal coverings. Thus, by Theorem \ref{pakkanen}, the maps $p$ and $q$ are completed coverings. Since $q=f \circ p$ and $p$ and $q$ are completed coverings, we have by Lemma \ref{sina}, that the map $f$ is a completed covering.
\end{proof}

\section{Existence of monodromy representations}

\subsection{Completed normal coverings and orbit maps}\label{ooo}

In this section we will prove Theorems \ref{hilu} and \ref{ll} in the introduction.
 
\begin{Thm}\label{vauva}Let $g \colon Y' \to Z'$ be a normal covering onto a large subset $Z'\subset Z$ of a manifold $Z,$ $f \colon Y \ra Z$ the completion of $g \colon Y' \ra Z$ and $\T(f)$ the deck-transformation group of $f.$ Then for the canonical maps $\pi$ and $\phi$
\begin{equation*}
\xymatrix{
Y \ar[dr]_{\pi} \ar[rr]^{f} && Z\\
&Y/\T(f) \ar[ur]_{\phi} &}
\end{equation*}
the map $\phi$ is a homeomorphism if and only if $f$ is discrete.
In particular, $f$ is an orbit map if and only if $f$ is discrete.
\end{Thm}

\begin{Not}\label{huu}Let $g \colon Y' \to Z'$ be a normal covering onto a large subset $Z'\subset Z$ of a manifold $Z$ and $f \colon Y \ra Z$ the completion of $g \colon Y' \ra Z.$ Then the deck transformations of $g$ are restrictions of deck transformations of $f$ by Lemma \ref{biotin}, that is,
$\T(g)=\{\tau \r Y' \colon Y' \ra Y' : \tau \in \T(f)\},$ and each $\tau \in \T(f)$ is the completion of $\tau \r Y' \colon Y' \ra Y$ since $Y' \subset Y$ is large. In particular, the homomorphism
$$\varphi \colon \T(f) \ra \T(g), \tau \mapsto \tau \r Y',$$ is an isomorphism by the uniqueness of completions, see \cite[Cor. 7.8]{MA}.
\end{Not}

\begin{proof}[Proof of Theorem \ref{vauva}] Suppose first that $f$ is discrete. Since $\phi$ is open and continuous, it is sufficient to show that it is bijective; we need to show that for every $z \in Z$ and $y_1, y_2 \in f^{-1}\{z\}$ there exists $\tau \in \T(f)$ so that $\tau(y_1)=y_2.$ 

Since $f$ is discrete, there exists $W \in \mathcal{N}_Z(z)$ so that the $y_i$-components $V_i$ of $f^{-1}(W)$ satisfy $f^{-1}\{z\} \cap V_i=\{y_i\}$ for $i \in \{1,2\}.$ Since $g$ is a normal covering and $Y' \subset Y$ is large, there exists $\tau' \colon Y' \ra Y'$ in $\T(g)$ for which $\tau'(Y' \cap V_1)=Y' \cap V_2.$ The completion $\tau \in \T(f)$ of $\tau' \colon Y' \ra Y$ satisfies $\tau(y_1)=y_2.$ 

Suppose now that $\phi$ is a homeomorphism. Thus $\T(f)$ acts transitively on $f^{-1}\{z\}$ for every $z \in Z.$ Arguing towards contradiction assume $f$ is not discrete. Then there is a point $y \in Y$ so that $f^{-1}\{f(y)\} \cap U$ is an infinite set for every neighbourhood $U$ of $y.$ Let $z:=f(y).$ Since $g \colon Y' \ra Z'$ is a covering between manifolds, $\T(g)$ is countable. Since $\T(g)\cong\T(f),$ for a contradiction it is sufficient to show that $f^{-1}\{z\}$ is uncountable. 

Towards a contradiction suppose that the set $f^{-1}\{z\}$ is countable and let $f^{-1}\{z\}=\{y_0,y_1, \ldots \}$ be an enumeration of $f^{-1}\{z\}.$ Then for every $k \in \N$ and neighbourhood $U$ of $y_k$ the set $f^{-1}\{z\} \cap U$ is infinite, since $f$ is an orbit map. 

Fix $W_0 \in \mathcal{N}_Z(z)$ and let $U_0$ be the $y_0$-component of $f^{-1}(W_0).$ Then fix a point $y'_1:=y_k \in \{y_1, y_2 \ldots \} \cap U_0$ for which $y_i \notin U_0, i < k,$ and let $W_1 \in \mathcal{N}_Z(z)$ be such that the $y'_1$-component $U_1$ of $f^{-1}(W_1)$ satisfies $U_0 \supset U_1$ but $y_0 \notin U_1.$ Proceed in this fashion to construct sequences $W_0 \supset W_1 \supset \cdots$ and $U_0 \supset U_1 \supset \cdots $ so that $\Cap_{n=0}^{\infty}W_n=\{z\}.$ 

Since $f$ is a complete spread, there exists a point $y_{\infty}$ so that $\Cap_{n=0}^{\infty}U_n=\{y_{\infty}\}.$ However, by the construction, $\Cap_{n=0}^{\infty}U_n \cap f^{-1}\{z\}$ is empty. This is a contradiction. Hence $f^{-1}\{z\}$ is uncountable and $\phi$ is not injective. Since $\phi$ is bijective, this is a contradiction and we conclude that $f$ is discrete.
\end{proof}

Theorems \ref{hilu} and \ref{ll} in the introduction are now easy corollaries of Theorem \ref{vauva}.
\begin{proof}[Proof of Theorem \ref{hilu}] Let $p \colon Y \ra Y$ be a completed normal covering. Then there are large subsets $Y'\subset Y$ and $Z' \subset Z$ so that $p \r Y' \colon Y' \ra Z'$ is a normal covering and $p$ is the completion of $p \r Y' \colon Y' \ra Z.$ By Theorem \ref{vauva} $p$ is an orbit map if and only if $p$ is a discrete map. 
\end{proof}

\begin{proof}[Proof of Theorem \ref{ll}] Let $f \colon X \ra Z$ be a completed covering. Suppose $f$ has a monodromy representation $(Y,p,q).$ We need to show that $f,$ $p$ and $q$ are discrete maps. Since $p$ and $q$ are completed coverings and orbit maps, $p$ and $q$ are completed normal coverings. Thus $p$ and $q$ are discrete maps by Theorem \ref{vauva}. Since $p$ is an open surjection and $q=p \circ f,$ the discreteness of $f$ follows from the discreteness of $q.$
\end{proof}

\subsection{Image of the branch set of a completed covering}\label{ando}
In this section we prove Theorem \ref{yopo} in the introduction. We recall that the branch set $B_f$ of a completed covering $f \colon X \ra Z$ is the set of points where $f$ fails to be a local homeomorphism. We say that a completed covering $f \colon X \ra Z$ is \textit{uniformly discrete} if every $z \in Z$ has a neighbourhood $U$ so that $f^{-1}\{z\} \cap D$ is a point for every component $D$ of $f^{-1}(U).$  

\begin{Prop}\label{lamppu} Let $p \colon Y \ra Z$ be a discrete completed normal covering. Then $p(B_p) \subset Z$ is closed for the branch set $B_p$ of $p.$ 
\end{Prop}

\begin{proof}Let $U:=Y \setminus p^{-1}(p(B_p)).$ Since $p$ is an orbit map by Theorem \ref{vauva}, we have $p^{-1}(p(B_p))=B_p$ and $p(B_p)=Z \setminus p(U).$ Now $U \subset Y$ is open, since $B_p \subset Y$ is closed. Thus $p(B_p)=Z \setminus p(U)$ is closed, since $p$ is an open map. 
\end{proof}

\begin{Prop}\label{gingerale} Let $f \colon X \ra Z$ be a completed covering between manifolds. Suppose there are uniformly discrete completed normal coverings $p \colon Y \ra X$ and $q \colon Y \ra Z$ satisfying $q:=f \circ p.$ Then $f(B_f) \subset Z$ is closed. 
\end{Prop}

For proving Proposition \ref{gingerale} we need the following lemma.

\begin{Apulause}\label{kylajuhlat} Let $p \colon Y \ra Z$ be a uniformly discrete completed normal covering. Let $z \in Z$ and $U, V  \in \mathcal{N}_Z(z)$ be such that $U \subset V$ and $p^{-1}\{z\} \cap D$ is a point for every component $D$ of $p^{-1}(V).$ Let $D$ be a component of $p^{-1}(V)$ and $C$ a component of $p^{-1}(U)$ so that $C \subset D.$ Then the groups
$$H_D:=\{ \tau \in \T(p) : \tau(D)=D\} \text{ and }H_C:=\{ \tau \in \T(p) : \tau(C)=C\}$$
satisfy $H_D=H_C.$
\end{Apulause}
\begin{proof} Let $Y' \subset Y$ and $Z' \subset Z$ be large subsets so that $g:=p \r Y' \colon Y' \ra Z'$ is a normal covering. Then $D \cap Y'$ is a component of $V \cap Z'$ and $C \cap Y'$ is a component of $U \cap Z'.$ By Theorem \ref{vauva} $p$ is an orbit map, and in particular, the map $\T(p) \ra \T(g), \tau \mapsto \tau \r Y'$ is an isomorphism \cite[Cor. 7.8]{MA}. Thus it is sufficient to show that $H'_D=H_C'$ for the groups 
$$H'_D:=\{ \tau \in \T(g) : \tau(D \cap Y')=D \cap Y'\}$$ and $$H'_C:=\{ \tau \in \T(g) : \tau(C \cap Y')=C \cap Y'\}.$$
Let $z' \in Z' \cap U.$ Since $p^{-1}\{z\} \cap D$ is a point, we have $p^{-1}(U) \cap D=C$ as a consequence of Lemma \ref{brunate}. Thus $g^{-1}\{z'\} \cap (D \cap Y')=g^{-1}\{z'\} \cap (C \cap Y').$ Let $y' \in g^{-1}\{z'\} \cap (C \cap Y')$ and let
$\tau_{y} \colon Y' \ra Y'$ be for every $y \in p^{-1}\{z'\}$
the unique deck-transformation of $g$ that satisfies $\tau_y(y')=y.$
Then
\begin{align*}H'_{D}&=\{ \tau_y : y \in g^{-1}\{z'\} \cap (D \cap Y')\}
\\&=\{ \tau_y : y \in g^{-1}\{z'\} \cap (C \cap Y')\}=H'_C.
\end{align*}
Thus $H_D=H_C.$
\end{proof}

\begin{proof}[Proof of Proposition \ref{gingerale}] Let $z \in Z \setminus f(B_f)$ and $U \in \mathcal{N}_Z(z)$ be a neighbourhood of $z$ so that the set $q^{-1}\{z\} \cap D$ is a point for every component $D$ of $q^{-1}(U)$.  We note that, since $p \colon Y \ra X$ is onto $X$, $U\subset Z \setminus f(B_f)$ if $f \r p(D)$ is an injection for every component $D$ of $q^{-1}(U).$ 

Let $D$ be a component of $q^{-1}(U)$ and
$H_D:=\{\tau \in \T(q) : \tau(D)=D\}.$ By Theorem \ref{vauva}, $q$ is an orbit map. Thus $f \r p(D)$ is an injection, if $H_D \subset \T(p).$ 

Let $y \in q^{-1}\{z\} \cap D$ and $x:=p(y).$ Since $x \in f^{-1}(Z \setminus B_f)$ there exists $U_0 \in \mathcal{N}_Z(z)$ so that $f \r V$ is an injection for the $x$-component $V$ of $f^{-1}(U_0)$ and so that $U_0 \subset U.$ Let $C$ be the $y$-component of $q^{-1}(U_0).$ Then $f \r p(C)=f \r V$ by Lemma \ref{brunate}. Thus $H_C:=\{\tau \in \T(q) : \tau(C)=C\} \subset \T(p),$ since $f \r p(C)$ is an injection.

By Lemma \ref{kylajuhlat}, $H_D=H_C,$ since $C \subset D.$ Thus $H_D=H_C \subset \T(p).$ Thus $f \r p(D)$ is an injection. Thus $U \subset Z \setminus f(B_f)$ and $f(B_f)\subset Z$ is closed. 
\end{proof}

Towards proving Theorem \ref{yopo} we make the following observation.

\begin{Apulause}\label{spread_5}Let $p \colon Y \to Z$ be a completed normal covering. Then $p$ is a discrete map if and only if $p$ is uniformly discrete.
\end{Apulause}
\begin{proof} Suppose $p$ is discrete. Then, by Theorem \ref{vauva}, the map $p$ is an orbit map. Let $y \in Y$ and $z:=p(y).$ Then since $p$ is both a discrete map and a spread, there exists a neighbourhood $V$ of $y$ so that $V$ is the $y$-component of $p^{-1}(p(V))$ and $V \cap p^{-1}\{p(y)\}=\{y\}.$ Let $U$ be a component of $p^{-1}(p(V)).$ By Lemma \ref{vindpark} there is a deck-transformation $\tau \colon Y \ra Y$ of $p$ so that 
$\tau(V)=U.$ Thus $$p^{-1}\{p(y)\} \cap U=\{\tau(y)\}$$
and $p$ is uniformly discrete.

Suppose now that every $z \in Z$ has a neighbourhood $W$ such that $p^{-1}\{z\} \cap V$ is a point for each component $V$ of $p^{-1}(W).$ Then $p$ is discrete, since $p$ is a spread. 
\end{proof}

\begin{proof}[Proof of Theorem \ref{yopo}] Let $f \colon X \ra Z$ be a completed covering between ma\-nifolds and $(Y,p,q)$ a monodromy representation of $f.$ %We need to show that $f(B_p) \subset X,$ $q(B_q) \subset Z$ and $f(B_f) \subset Z$ are closed subsets for the branch sets $B_p,$ $B_q$ and $B_f.$ 
As a consequence of Theorem \ref{vauva} and Lemma \ref{spread_5} $p \colon Y \ra X$ and $q \colon Y \ra Z$ are uniformly discrete completed normal coverings. By Proposition \ref{lamppu}, $p(B_p) \subset X$ and $q(B_q) \subset Z$ are closed. By Proposition \ref{gingerale}, $f(B_f) \subset Z$ is closed, since $q=f \circ p.$ 
\end{proof}

\subsection{Discrete completed normal coverings}\label{dcc} 

In this section we prove Theorems \ref{mm} and \ref{relas} in the introduction. Let $g \colon Y' \to Z'$ be a normal covering onto a large subset $Z'\subset Z$ of a PL manifold $Z,$ $p \colon Y \ra Z$ the completion of $g \colon Y' \ra Z,$ $z_0 \in Z'$ and $x_0 \in g^{-1}\{z_0\}.$ In general $p$ is not discrete, see \cite[Ex.\;10.6.]{MA}, but if $p$ has locally finite multiplicity, then $p$ is discrete, see \cite[Thm.\;9.14.]{MA}. The discreteness of the map $p$ depends on the properties of the group $\pi(Z',z_0)/g_*(\pi(X',x_0))$ and the manifold $Z.$ To capture this relation we define stability of $g \colon Y' \ra Z$ and show that $p \colon Y \ra Z$ is discrete if and only if $g \colon Y' \ra Z$ is stable. 

Let $\mathcal N_Z(z;z_0)$ be the set of open and connected neighbourhoods $U \subset Z$ of $z$ satisfying $z_0 \in U.$ We denote by $\iota_{U,Z'}$ the inclusion $U \hra Z'.$ 

\begin{Def} Let $Z$ be a PL manifold, $Z'\subset Z$ a large subset, $z_0 \in Z'$ and $N \unlhd \pi(Z',z_0)$ a normal subgroup. A subset $U \in \mathcal{N}_Z(z;z_0)$ is a $(Z',N)$\textit{-stable} neighbourhood of $z$ if $$\mathrm{Im}(\pi \circ (\iota_{V \cap Z',Z'})_*)=\mathrm{Im}(\pi \circ (\iota_{U \cap Z',Z'})_*)$$
for every $V \in \mathcal{N}_Z(z;z_0)$ contained in $U;$ 

\begin{displaymath}
\xymatrix{\pi(V \cap Z',z_0) \ar[dr]_{(\iota_{V \cap Z',Z'})_*} && \pi(U \cap Z',z_0)  \ar[dl]^{(\iota_{U \cap Z',Z'})_*}\\
&\pi(Z',z_0)\ar[d]_{\pi}&\\
&\pi(Z',z_0)/N,&
}\end{displaymath}
where $\pi \colon \pi(Z',z_0) \ra \pi(Z',z_0)/N$ is the canonical quotient map.
\end{Def}

\begin{Def}\label{stability} Let $Z$ be a PL manifold and $Z' \subset Z$ a large subset, $z_0 \in Z'$ and $N \unlhd \pi(Z',z_0)$ a normal subgroup. The manifold $Z$ is $(Z',N)$-\textit{stable} if every point $z \in Z$ has a $(Z',N)$-stable neighbourhood. 
\end{Def}

\begin{Not}\label{stabhuomio}Let $Z$ be a PL manifold, $Z' \subset Z$ a large subset, $z_0 \in Z'$ and $N'$ and $N$ normal subgroups of $\pi(Z',z_0)$ so that $N' \subset N.$ Then $Z$ is $(Z',N)$-stable if $Z$ is $(Z',N')$-stable. Indeed, if $U \in \mathcal{N}_Z(z,z_0)$ is a $(Z',N')$-stable neighbourhood of $z,$ then $U$ is a $(Z',N)$-stable neighbourhood of $z.$ 
\end{Not}

\begin{Def} A mapping $g \colon Y' \ra Z$ is \textit{stable} if $g \colon Y' \ra g(Y')$ is a covering, $g(Y') \subset Z$ is large and $Z$ is $(g(Y'),\mathrm{Ker}(\sigma_g))$-stable for the monodromy $\sigma_g$ of the covering $g.$
\end{Def}

We say that a completed covering $f \colon Y \ra Z$ is a \textit{stabily completed covering} if $f \r Y' \colon Y' \ra Z$ is stable for a large subset $Y' \subset Y.$ We say that a completed covering $f \colon Y \ra Z$ is a \textit{stabily completed normal covering}, if $g:=f \r Y' \colon Y' \ra Z$ is stable and $g \colon Y' \ra g(Y')$ a normal covering for a large subset $Y' \subset Y.$ 

\begin{Thm}\label{TFAE} Let $p \colon Y \ra Z$ be a completed normal covering, $Y' \subset Y$ and $Z' \subset Z$ large subsets so that $g:=p \r Y' \colon Y' \ra Z'$ is a normal covering, $z_0 \in Z'$ and $y_0 \in g^{-1}\{z_0\}.$ Then the following conditions are equivalent:
\begin{itemize}
\item[(a)] $p\colon Y \ra Z$ is discrete,
\item[(b)] $p\colon Y \ra Z$ is uniformly discrete and 
\item[(c)] $g \colon Y' \ra Z$ is stable.
\end{itemize}
\end{Thm}
The proof consists of three parts. Lemma \ref{spread_5} proves $(a) \Leftrightarrow (b),$ Proposition \ref{world} proves $(c)\Rightarrow(b)$ and Proposition \ref{somac} proves $(b)\Rightarrow(c).$

Next we are going to state two lemmas used in proving Proposition \ref{world} and Proposition \ref{somac}. 

\begin{Apulause}\label{D_1}Let $Z$ be a manifold, $Z' \subset Z$ large and $U$ and $V$ simply-connected open subsets of $Z$ so that $U \subset Z'$ and the set $U \cap V$ is a non-empty path-connected set. Let $z_0 \in U$ and $z_1\in U \cap V \cap Z'.$ Then for every path $\gamma \colon z_0 \cu z_1$ in $(U \cup V) \cap Z'$ and every loop $\al \colon (S_1,e_0) \ra ((U \cup V) \cap Z', z_0)$ there exists a loop $\be \colon (S^1,e_0) \ra (V \cap Z',z_1)$ so that $[\al]=[\gamma\be\gamma^{\law}]$ in the group $\pi((U \cup V) \cap Z',z_0).$
\end{Apulause}
\begin{proof}We denote $W:=U \cup V.$ Let $\gamma \colon z_0 \cu z_1$ be a path in $W \cap Z'$ and $\al\colon (S^1,e_0) \ra (W \cap Z',z_0)$ a loop. Since $U \subset Z',$ we have $W \cap Z'=U \cup (V \cap Z').$ Moreover, the sets $U$, $V \cap Z'$ and $U \cap (V \cap Z')$ are path-connected open subsets of $W \cap Z',$ since $Z' \subset Z$ is large. 
Since $U$ is simply-connected, the homomorphism $\iota_* \colon \pi(V \cap Z',z_1) \ra \pi(W \cap Z',z_1)$ induced by the inclusion $\iota \colon V \cap Z' \ra W \cap Z'$ is by the Seifert-van Kampen Theorem an epimorphism, see \cite[Thm 4.1.]{Mas}. Hence there exist a loop $\be \colon (S^1,e_0) \ra (V \cap Z',z_1)$ so that 
$[\be]=[\gamma^{\law}\al\gamma]$ in $\pi(W \cap Z',z_1).$ Now
$[\al]=[\gamma\be\gamma^{\law}]$ in $\pi(W \cap Z',z_0).$
\end{proof}

For the terminology of simplicial structures we refer to Hatcher \cite[Ch.\;2]{H}. We denote by $\sigma(v_0,\ldots, v_n)$ the simplex in $\R^k$ spanned by vertices $v_0, \ldots, v_n \in \R^k.$ The open simplex $\mathrm{Int}(\sigma):=\sigma \setminus \partial(\sigma)$ we call the simplicial interior of $\sigma.$ For a simplicial complex $K$ consisting of simplices in $\R^k$ the polyhedron is a subset $|K| \subset \R^k.$ For $z \in |K|$ we call $\mathrm{St}(z):=\Cup_{\sigma \in K; z \in \sigma} \mathrm{Int}(\sigma)$ the star neighbourhood $z$ (with respect to $K$). 

A \textit{triangulation} of a manifold $Z$ is a simplicial complex $T$ satisfying $Z \approx |T|.$ We say that a manifold is a \textit{PL manifold} if it has a triangulation. If $Z$ is a PL manifold and $T$ a triangulation of $Z$ with simplexes in $\R^k,$ then we assume from now on that
$$Z=|T| \subset \R^k.$$

Let $K$ be a simplical complex. We will introduce for the polyhedron $|K|$ a topological basis related to the star neighbourhoods $\mathrm{St}(x), x \in |K|,$ and show a technical property for this particular basis. Let $x_0 \in |K|$ be a point and $t \in (0,1].$ For a simplex $\sigma=\sigma(v_0, \ldots, v_m) \in K$ for which $x_0 \in \sigma$ we denote by $\tau(\sigma,x_0,t)$ the simplex spanned by the vertices $v'_j=x_0+t(v_j-x_0), j \in \{0, \ldots, m\}.$ The $t$-\textit{star} of $x_0$ is  the set
$$\mathrm{St}_t(x_0)= \Cup_{\{\sigma \in K : x_0 \in \sigma\}} \mathrm{Int}(\tau(\sigma,x_0,t)),$$
where $\mathrm{Int}(\tau(\sigma,x_0,t))$ is the simplicial interior of $\tau(\sigma,x_0,t)$ for all $\sigma \in K$ satisfying $x_0 \in \sigma.$

We note that the collection
$(\mathrm{St}_{1/n}(x_0))_{n \in \N}$ forms a neighbourhood basis of the point $x_0 \in |K|,$ and that there exists a contraction 
$$H \colon \mathrm{St}(x_0) \times [0,1] \ra \mathrm{St}(x_0)$$ of the open star $\mathrm{St}(x_0)$ to $x_0$ for which $H(\mathrm{St}(x_0)  \times \{1/n\})=\mathrm{St}_{1/n}(x_0)$ for every $n \in \N$ and 
$H(\{x\} \times [0,1])\subset [x,x_0]$ for all $x \in \mathrm{St}(x_0).$ 

\begin{Apulause}\label{D_2}Let $Z$ be a PL manifold, $Z' \subset Z$ large and $T$ a triangulation of $Z$ and $z_0 \in Z'$ a vertex in $T.$ Let $z \in Z$ be a point and $\mathrm{St}_{1/m}(z)$ the $1/m$-star of $z$ for $m>1.$ Suppose $X \subset Z$ is an open connected satisfying $z_0\in X$ and $\mathrm{St}_{1/m}(z) \subset X.$ Then there exists a simply-connected open set $U \subset X \cap Z'$ so that $z_0 \in U$ and that the set $U \cap \mathrm{St}_{1/m}(z)$ is non-empty and path-connected. 
\end{Apulause}
\begin{proof} Suppose $z_0 \in \overline{\mathrm{St}_{1/m}(z)}.$ Since $z_0$ is a vertex of $T$ and $m>1,$ we have $z=z_0.$ Thus $z=z_0 \in \mathrm{St}_{1/m}(z) \cap Z'.$ Since $\mathrm{St}_{1/m}(z)\cap Z'  \subset Z$ is open, there exists a simply-connected neighbourhood $U\subset \mathrm{St}_{1/m}(z) \cap Z'$ of $z.$ This proves the claim in the case $z_0 \in \overline{\mathrm{St}_{1/m}(z)}.$ 

Suppose then that $z_0 \notin \overline{\mathrm{St}_{1/m}(z)}.$ Since $Z'\subset Z$ is large and $\mathrm{St}_{1/m}(z) \subsetneq X,
$ there exists a point $z_1 \in \partial (\mathrm{St}_{1/m}(z)) \cap X \cap Z',$ that is a simplicial interior point of an $n$-simplex $\sigma \in T$ and a face $f$ of $\tau(\sigma,z,1/m).$ Since $|T|\subset \R^k$ for some $k \in \N,$ the Euclidean metric of $\R^k$ induces a metric on $|T|$. Let $$r_1:=\text{min}\{\text{dist}(z_1, f) \mid f \text{ is a face of } \tau(\sigma,z,1/m) \text{ and } z_1 \notin f\}$$ and
$$r_2=\text{dist}(z_1, Z \setminus (Z'\cap X )).$$ 
Then $r:=\text{min}\{r_1,r_2\}>0.$

We fix $k \in \N$ and a $1$-subcomplex $L$ in the $k^{th}$ barycentric subdivision $\mathrm{Bd}^k(T)$ of $T$ so that $|L|\subset |T|$ is a simple closed curve in $X \cap Z'\cap (Z \setminus \overline{\mathrm{St}_{1/m}(z)})$ from $z_0$ to $\mathrm{Int}(\sigma)$ satisfying $\text{dist}(z_1,|L| \cap \sigma) <r/4.$ 

We fix $l\geq k+2$ large enough so that for 
$$W:=\Cup_{z' \in |L|} \mathrm{St}^{l}(z'),$$
where $\mathrm{St}^{l}(z')$ is for every $z' \in |L|$ the open star of $z'$ in $\mathrm{Bd}^l(T),$ we have 
$$\overline{W} \subset X \cap Z' \cap (Z \setminus \overline{\mathrm{St}_{1/m}(z)}).$$
Then $|L|$ is a strong deformation retract of $W$ since $l\geq k+2;$ see regular neighbourhoods in \cite[Sec.\;2]{RS}. In particular, $W$ is simply-connected since $|L|$ is simply connected. 

Let $K < \mathrm{Bd}^l(T)$ be the $(n-1)$-subcomplex satisfying $\partial W=|K|$ and let $S$ be the collection of $(n-1)$-simplexes $\tau \in K$ for which 
$$\text{dist}(z_1, \overline{W} \cap \sigma)=\text{dist}(z_1, \tau).$$
Since $z_1$ is an interior point of $f$ and $\text{dist}(z_1,\overline{W} \cap \sigma)<r_1,$ there exists $\tau \in S$ and $z_2 \in \mathrm{Int}(\tau)$ for which the line segment $[z_1,z_2] \subset \mathrm{Int}(\sigma)$ has length at most $r/2$
and
$$[z_1,z_2] \cap (\mathrm{St}_{1/m}(z) \cup W)=\{z_1,z_2\}.$$ 
 
Now, for all $t<r/2$ the $t$-neighbourhood $W_t \subset \mathrm{Int}(\sigma)$ 
of $[z_1,z_2]$ satisfies $W_t \subset X \cap Z'$ and the intersections $W \cap 
W_t$ and $\mathrm{St}_{1/m}(z) \cap W_t$ are non-empty. Since $z_2 \in \mathrm{Int}(\tau),$ the intersections $W \cap W_{t_0}$ and $\mathrm{St}_{1/m}(z) \cap W_{t_0}$ are non-empty intersections of two convex sets for some $t_0<r/2$. Thus $W \cap W_{t_0}$ and $\mathrm{St}_{1/m}(z) \cap W_{t_0}$ are path-connected. Since the sets $W_{t_0}$ and $W$ are simply-connected and the set $W \cap W_{t_0}$ path-connected, the set $U=W \cup W_{t_0} \subset Z'$ is simply-connected by the Seifert-van Kampen Theorem. Further, $z_0 \in W \subset U.$ Since $\mathrm{St}_{1/m}(z) \cap W_{t_0}$ is path-connected and $U \cap \mathrm{St}_{1/m}(z)=\mathrm{St}_{1/m}(z) \cap W_{t_0},$ this proves the claim. 
\end{proof}

\begin{Prop}\label{world}Let $p \colon Y \ra Z$ be a completed normal covering, $Y' \subset Y$ and $Z' \subset Z$ large subsets so that $g:=p \r Y' \colon Y' \ra Z'$ is a normal covering. If $g \colon Y' \ra Z$ is stable, then $p \colon Y \to Z$ is uniformly discrete.
\end{Prop}

\begin{proof}Let $z_0 \in Z'$ and $y_0 \in g^{-1}\{z_0\}.$ Let $z \in Z.$ Since $g$ is stable the manifold $Z$ is $(Z',g_*(\pi(Y',y_0))$-stable. Let $T$ be a triangulation of $Z$ having $z_0$ a vertex. Let $V$ be a $(Z',g_*(\pi(Y',y_0)))$-stable neighbourhood of $z$ and $D$ a component of $p^{-1}(V).$ It is now sufficient to show that $p^{-1}\{z\} \cap D$ is a point. By Lemma \ref{brunate}, we have $p(D)=V,$ and hence $p^{-1}\{z\} \cap D\neq \es.$ 

We suppose towards a contradiction that there are distinct points $d_1$ and $d_2$ in $p^{-1}\{z\} \cap D.$ Using points $d_1$ and $d_2,$ we construct a loop  
$\al_1 \colon (S^1,e_0) \ra (V \cap Z',z_0)$ and $V_1 \in \mathcal{N}_Z(z,z_0)$ so that $V_1 \subset V$ and 
$$[\al_1]g_*(\pi(Y',y_0)) \notin \mathrm{Im}(\pi \circ \iota_{Z' \cap V_1,Z'})).$$
This contradicts the $(Z',g_*(\pi(Y',y_0)))$-stability of $Z$ and proves the claim.

We begin by choosing the neighbourhood $V_1 \subset V$ of $z.$ For all $n \in \N$ let $\mathrm{St}_{1/n}(z)$ be the $1/n$-star of $z$ in the triangulation $T.$ Then $(\mathrm{St}_{1/n}(z))_{n \in \N}$ is a neighbourhood basis for $z.$ Since $p \colon Y \to Z$ is a spread and $Y$ a Hausdorff space, there exists $n \geq 2$ so that $\mathrm{St}_{1/n}(z) \subset V$ and the $d_i$-components $C_i,$ $i \in \{1,2\},$ of $p^{-1}(\mathrm{St}_{1/n}(z))$ are disjoint. By Lemma \ref{D_2}, there is such an open simply-connected set $W \subset V \cap Z'$ containing $z_0$ that the sets $W \cup \mathrm{St}_{1/n}(z)$ and $W \cap \mathrm{St}_{1/n}(z)$ are path-connected. Hence  
$$V_1=W \cup \mathrm{St}_{1/n}(z) \subset V$$ 
is an open connected neighbourhood of $z$ containing $z_0$. Configuration illustrated in Figure \ref{wer}.  
\begin{figure}[ht]
\includegraphics{Martina.1}
\caption{}\label{wer}
\end{figure}

We continue by choosing a loop $\al_1 \colon (S^1,e_0) \ra (V \cap Z',z_0)$. Fix $z_1 \in W \cap \mathrm{St}_{1/n}(z) \cap Z'.$ Since $W$ is a connected set containing points $z_0$ and $z_1,$ there is a path $\gamma \colon z_0 \cu z_1$ in $W.$ Since $\mathrm{St}_{1/n}(z)$ is connected, Lemma \ref{brunate} implies that there are points $c_i \in p^{-1}\{z_1\} \cap C_i$ for $i \in \{1,2\}.$ By Lemma \ref{biotin}, $p^{-1}(Z')=Y'.$ Thus $c_1, c_2 \in Y' \cap D.$ Since $Y' \subset Y$ is large, the set $Y \setminus Y'$ does not locally separate $Y$ and the set $D \cap Y'$ is path-connected. Let $\be \colon [0,1] \ra D \cap Y'$ be a path $\be \colon c_1 \cu c_2.$ Then $g \circ \be \colon [0,1] \ra V \cap Z'$ is a loop at $z_1.$ 
We set $\al_1=\gamma(g \circ \be) \gamma^{\leftarrow}.$

Suppose there is a loop $\al_2 \colon (S^1,e_0) \ra (V_1 \cap Z',z_0)$ satisfying
$$[\al_2] \in [\gamma(g \circ \be) \gamma^{\leftarrow}]g_*(\pi(Y',y_0)).$$
Since $W \subset Z'$ and $\gamma$ is a path in $W \subset V_1 \cap Z',$ there is by Lemma \ref{D_1} a loop $\al_3 \colon (S^1,e_0) \ra (\mathrm{St}_{1/n}(z) \cap Z',z_1)$ satisfying 
$$[\al_2]=[\gamma \al_3 \gamma^{\leftarrow}].$$

Next we change the base point of $Y$ as follows. Since $g \colon Y' \ra Z'$ is a normal covering and $W \subset Z',$ the path $\gamma^{\leftarrow} \colon [0,1] \ra W$ has a lift $\ol{\gamma^{\leftarrow}}_{c_1} \colon [0,1]\ra Y'.$ We set $y_1:=\ol{\gamma^{\leftarrow}}_{c_1}(1).$ 
Since $g \colon Y' \ra Z'$ is a normal covering, we have 
$$g_*(\pi(Y',y_1))=g_*(\pi(Y',y_0)).$$ Hence $$[\gamma \al_3 \gamma^{\leftarrow}] \in [\gamma(g \circ \be) \gamma^{\leftarrow}]g_*(\pi(Y',y_0))=[\gamma(g \circ \be) \gamma^{\leftarrow}]g_*(\pi(Y',y_1)).$$ 
Thus the lifts $\ol{\gamma \al_3 \gamma^{\leftarrow}}_{y_1}$ and $\ol{\gamma(g \circ \be) \gamma^{\leftarrow}}_{y_1}$ have a common endpoint. By the uniqueness of lifts,
$$\ol{\al_3}_{c_1}(1)=\ol{(g \circ \be)}_{c_1}(1)=\be(1)=c_2.$$ 

On the other hand, since $\al_3$ is a path in $\mathrm{St}_{1/n}(z),$ the lift $\ol{\al_3}_{c_1}$ is a path in $p^{-1}(\mathrm{St}_{1/n}(z)).$ Since $C_1$ is a component of $p^{-1}(\mathrm{St}_{1/n}(z)),$ this implies that $\ol{\al_3}_{c_1}[0,1] \subset C_1.$ Thus $c_2=\ol{\al_3}_{c_1}(1) \in C_1,$ which is a contradiction. Hence $p^{-1}\{p(y)\} \cap D$ is a point. 
\end{proof}

\begin{Prop}\label{somac}
Let $p \colon Y \ra Z$ be a completed normal covering, $Y' \subset Y$ and $Z' \subset Z$ large subsets so that $g:=p \r Y' \colon Y' \ra Z'$ is a normal covering. If $p$ is uniformly discrete, then $g \colon Y' \ra Z$ is stable.  
\end{Prop}
\begin{proof}Let $z_0 \in Z',$ $y_0 \in g^{-1}\{z_0\}$ and $T$ a triangulation of $Z$ so that $z_0$ is a vertex. We prove the statement by constructing a $(Z',g_*(\pi(Y',y_0))$-stable neighbourhood for every $z \in Z.$ 

Suppose that $z=z_0.$ Since $p$ is uniformly discrete and $z_0 \in Z',$ there exists a simply-connected neighbourhood $V \subset Z'$ of $z$ so that $p^{-1}\{z\} \cap D$ is a point for each component $D$ of $p^{-1}(V).$ Then $V$ is a $(Z',g_*(\pi(Y',y_0))$-stable neighbourhood of $z,$ since $V \subset Z'$ and $V$ is simply-connected. 

Let then $z \in Z \setminus \{z_0\}.$ For all $n \in \N$ let $\mathrm{St}_{1/n}(z)$ be the $1/n$-star of $z$ in the triangulation $T.$ Since $(\mathrm{St}_{1/n}(z))_{n \in \N}$ is a neighbourhood basis for $z$ and $p$ is uniformly discrete, there exists $n \geq 2$ so that 
$p^{-1}\{z\} \cap D$ is a point for each component $D$ of $p^{-1}(\mathrm{St}_{1/n}(z)).$ By Lemma \ref{D_2}, there is an open simply-connected set $U \subset Z'$ containing $z_0$ so that $U \cup \mathrm{St}_{1/n}(z)$ and $U \cap \mathrm{St}_{1/n}(z)$ are path-connected. Let $V:=U \cup \mathrm{St}_{1/n}(z).$ 

Towards showing that $V$ is a $(Z',g_*(\pi(Y',y_0))$-stable neighbourhood of $z,$ let 
$\al \colon (S^1,e_0) \ra (V \cap Z',z_0)$ be a loop and let $V_1 \in \mathcal{N}_Z(z;z_0)$ be an open set contained in $V.$ It is sufficient to show that there exists a loop $\be \colon (S^1,e_0) \ra (V_1 \cap Z',z_0)$ satisfying
$[\be]g_*(\pi(Y',y_0))=[\al]g_*(\pi(Y',y_0)).$

We choose first a suitable representative from the class $[\al]$ as follows. Let $W$ be the $z$-component of $\mathrm{St}_{1/n}(z) \cap V_1.$ Since $Z' \subset Z$ is large, there exists a point $z_1 \in W \cap Z'$ and a path $\gamma \colon z_0 \cu z_1$ in $V_1 \cap Z'.$  By Lemma \ref{D_1}, there exists a loop $\al_1 \colon (S^1,e_0) \ra (\mathrm{St}_{1/n}(z) \cap Z' ,z_1)$ for which $[\al]=[\gamma \al_1 \gamma^{\leftarrow}]$ in $\pi(V \cap Z',z_0).$  

Denote the endpoint of the lift $\ol{\gamma_{y_0}}$ in $g$ by $y_1$ and the endpoint of the lift $\ol{\al_1}_{y_1}$ in $g$ by $y_2.$ Let $D_1$ be the $y_1$-component of $p^{-1}(\mathrm{St}_{1/n}(z)).$ We proceed by showing that there exists a loop $\al_2 \colon (S^1,e_0) \ra (W \cap Z',z_1)$ so that the lift $\ol{\al_2}_{y_1}$ in $g$ that begins at $y_1$ ends at $y_2.$

Let $C_1\subset D_1$ be the $y_1$-component of $p^{-1}(W).$ Since $W$ is a connected neighbourhood of $z,$ every component of $p^{-1}(W)$ intersects the set $p^{-1}\{z\}$ by Lemma \ref{brunate}. Since $W \subset \mathrm{St}_{1/n}(z)$ and the set $p^{-1}\{z\} \cap D_1$ is a point, this implies
$$p^{-1}(W) \cap D_1=C_1.$$
Hence the points $y_1$ and $y_2$ belong to $C_1 \cap Y'.$ Since $Y \setminus Y'$ does not locally separate $Y,$ 
there is a path $\be \colon y_1 \cu y_2$ in $C_1 \cap Y'.$ Let $\al_2:=g \circ \be.$ 

Since the lifts $\ol{\gamma \al_1 \gamma^{\leftarrow}}_{y_0}$ and $\ol{\gamma \al_2 \gamma^{\leftarrow}}_{y_0}$ end at the common point, 
$$[\gamma \al_2 \gamma^{\leftarrow}]\in [\gamma\al_1\gamma^{\leftarrow}]g_*(\pi(Y',y_0))= [\al]g_*(\pi(Y',y_0)).$$
Since $\gamma \al_2 \gamma^{\leftarrow}$ is a loop in $V_1 \cap Z'$ at $z_0,$ $V$ is a $(Z',g_*(\pi(Y',y_0)))$-stable neighbourhood of the point $z.$ Since this holds for every $z \in Z$ the manifold $Z$ is $(Z',g_*(\pi(Y',y_0)))$-stable.
\end{proof}

This concludes the proof of Theorem \ref{TFAE}. After the following two consequences of Theorem \ref{TFAE} we are ready to prove Theorems \ref{ilta}, \ref{mm} and \ref{relas} in the introduction. 

\begin{Cor}\label{viimeinen} Let $f \colon X \ra Z$ be a completed covering between PL manifolds, $X' \subset X$ and $Z' \subset Z$ large subsets and 
\begin{equation*}
\xymatrix{
& Y' \ar[ld]_{p'} \ar[rd]^{q'} &\\
X' \ar[rr]^{f'} & & Z' }
\end{equation*}
a commutative diagram of coverings so that $p'$ and $q'$ are normal coverings. Let $z_0 \in Z'$ and $y_0 \in q'^{-1}\{z_0\}$, $N=q'_*(\pi(Y',y_0))$ and
\begin{equation*}
\xymatrix{
& Y \ar[ld]_{p} \ar[rd]^{q} &\\
X \ar[rr]^{f} & & Z }
\end{equation*}
be a commutative diagram of open continuous maps so that $p$ is the completion of $p' \colon Y' \ra X$ and $q$ is the completion of $p' \colon Y' \ra Z.$ Then, the maps $p$ and $q$ are orbit maps if and only if $Z$ is $(Z',N)$-stable. 
\end{Cor}

Let $f \colon X \ra Z$ be a completed covering between manifolds and $q \colon Y \ra Z$ a completed covering that is an orbit map. We say that $q$ is \textit{natural of} $f$ if there are large subsets $X' \subset X,$ $Y' \subset Y$ and $Z' \subset Z$ so that $g:=f \r X'\colon X' \ra Z'$ and $h:=q \r Y' \colon Y' \ra Z'$ are coverings satisfying $\mathrm{Ker}(\sigma_g)=\mathrm{Ker}(\sigma_h).$ We recall that the monodromy group of $f$ is the monodromy group of $g.$

\begin{Thm}\label{hh} Let $f \colon X \ra Z$ be a stabily completed covering between PL manifolds. Then $f$ has a monodromy representation $(Y,p,q),$ where $q$ is natural of $f.$
\end{Thm}
\begin{proof} Let $X'\subset X$ and $Z' \subset Z$ be large subsets so that $g:=f \r X' \colon X' \ra Z'$ is a covering and $Z$ is $(Z',\mathrm{Ker}(\sigma_g))$-stable. By Theorem \ref{spread_3} there is a normal covering $p' \colon Y' \ra X'$ so that $q':=g \circ p'$ is a normal covering satisfying $q_*'(\pi(Y',y_0))=\mathrm{Ker}(\sigma_g).$ Thus 
$\mathrm{Ker}(\sigma_{q'})=q_*'(\pi(Y',y_0))=\mathrm{Ker}(\sigma_g)$ and the deck-transformation group $\T(q')$ is isomorphic to the monodromy group of $g.$ 

Let $p \colon Y \ra X$ be the completion of $p' \colon Y' \ra X$ and $q:=f\circ p \colon Y \ra Z.$ Then $q$ is the completion of $q' \colon Y' \ra Z$ by Proposition \ref{elama} and $q$ is natural of $f.$ Since $\T(q')$ is isomorphic to the monodromy group of $g,$ it is sufficient to show that $p$ and $q$ are orbit maps and $\T(q)\cong \T(q').$ Since $Z$ is $(Z',\mathrm{Ker}(\sigma_{q'}))$-stable, $q$ is discrete by Theorem \ref{TFAE}. Thus $q$ is an orbit map by Theorem \ref{vauva}, and $\T(q)\cong \T(q')$ by Remark \ref{huu}. Since $q=f \circ p$ is discrete, $p$ is discrete. Thus $p$ is an orbit map by Theorem \ref{vauva}.
\end{proof}

\begin{proof}[Proof of Theorem \ref{ilta}] Let $p \colon Y \ra Z$ be a completed normal covering onto a PL manifold $Z.$ Let $Y' \subset Y$ and $Z' \subset Z$ be large subsets so that $g:= p \r Y' \colon Y' \ra Z'$ is a normal covering. Since $p$ is discrete, $g \colon Y' \ra Z$ is stable by Theorem \ref{TFAE}. Thus $p$ is a stabily completed normal covering.
 
Suppose then that $p$ is a stabily completed normal covering. Then there are such large subsets $Y'' \subset Y$ and $Z'' \subset Z$ that $h:=p \r Y'' \colon Y'' \ra Z''$ is a normal covering and the map $h \colon Y'' \ra Z$ is stable. Thus the map $p$ is discrete by Theorem \ref{TFAE}.  
\end{proof}

\begin{proof}[Proof of Theorem \ref{mm}] Suppose $f \colon X \ra Z$ is a completed covering between PL manifolds and $(Y,p,q)$ a monodromy representation of $f.$ By Corollary \ref{viimeinen}, the maps $p$ and $q$ are stabily completed normal coverings. Since $q$ is a stabily completed covering, as a consequence of Theorem \ref{spread_3} and Remark \ref{stabhuomio}, also $f$ is a stabily completed covering.
\end{proof}

\begin{proof}[Proof of Theorem \ref{relas}] By Theorem \ref{mm} and Theorem \ref{hh} a completed co\-ve\-ring $f \colon X \ra Z$ between PL manifolds is a stabily completed covering if and only if $f$ has a monodromy representation $(Y,p,q).$ 
\end{proof}

We end this section with a corollary of Theorem \ref{viimeinen}. We recall that for a simplical $n$-complex $K,$ $K^{n-2}$ is the codimension $2$ skeleton of $K.$   

\begin{Cor}\label{simplisiaalinen} Let $f \colon X \ra Z$ be a surjective, open and discrete simplicial map from a PL $n$-manifold $X$ with triangulation $K$ onto a PL $n$-manifold $Z$ with triangulation $L.$ Let also $z_0 \in |L| \setminus |L^{n-2}|$ and $x_0 \in f^{-1}\{z_0\}.$ Then $g:=f \r |K| \setminus |K^{n-2}| \colon |K| \setminus |K^{n-2}| \ra |L| \setminus |L^{n-2}|$ is a covering and for all normal subgroups $N \unlhd \pi(|L| \setminus |L^{n-2}|,z_0)$ is contained in $\mathrm{Ker}(\sigma_g)$ for the monodromy $\sigma_g$ of $g,$ there is a commutative diagram 
\begin{equation*}
\xymatrix{
& Y \ar[ld]_{p} \ar[rd]^{q} &\\
X \ar[rr]^{f} & & Z }
\end{equation*}
of completed coverings so that $p$ and $q$ are orbit maps and the deck-trans\-for\-ma\-ti\-on groups $\T(p)$ and $\T(q)$ satisfy 
$$\T(p)\cong \pi(|K| \setminus |K^{n-2}|,x_0)/N \text{ and } \T(q) \cong \pi(|L| \setminus |L^{n-2}|,z_0)/N .$$  
\end{Cor}
\begin{proof} Since $f \colon X \ra Z$ is an open discrete simplicial map the branch set of $f$ is contained in $|K^{n-2}|.$ In particular $g:=f \r|K| \setminus |K^{n-2}| \colon |K| \setminus |K^{n-2}| \ra  |L| \setminus |L^{n-2}|$ is a covering and $f$ is the completion of $f \colon  |K| \setminus |K^{n-2}| \ra Z.$ Clearly, $Z$ is $(|L| \setminus |L^{n-2}|,\{e\})$-stable for the trivial group $\{e\} \subset \pi( |L| \setminus |L^{n-2}|,z_0).$ Further, $Z$ is $(|L| \setminus |L^{n-2}|, N )$-stable by Remark \ref{stabhuomio}. We conclude the statement from Theorem \ref{spread_3} and Theorem \ref{viimeinen}.
\end{proof}

\begin{Not}If the map $p \colon Y \ra X$ in Corollary \ref{simplisiaalinen} has locally finite multiplicity, then the space $Y$ is a locally finite simplicial complex and $p$ and $q$ are simplicial maps, see \cite[Sec.\;6]{F}. However, in the next section we show that $p$ need not have locally finite multiplicity. 
\end{Not}
%\begin{Proposition} Let $p \colon Y \ra Z$ be a completed covering with finite multiplicity, then $p$ is discrete.
%\end{Proposition}
%\begin{proof} Let $z \in Z$ and $y \in p^{-1}\{z\}.$ Since $p$ is a spread, there is for every $y' \in p^{-1}\{z\} \setminus \{y\}$ a set $W_{y'} \in \mathcal{N}_Z(z)$ so that the $y$-component of $p^{-1}[W_{y'}]$ does not contain the point $y_'$. Let $U:= \Cap_{y' \in p^{-1}\{z\} \setminus \{y\}} W_{y'}.$ Then the $y$-component $D$ of $p^{-1}[U]$ is a neighbourhood of $y$ and satisfies $D \cap p^{-1}\{z\}=\{y\}.$ Thus $p$ is discrete.  
%\end{proof}

\section{Existence of locally compact monodromy representations}

\subsection{Uniformly bounded local multiplicities} In this section we show the existence of a completed covering that has a monodromy representation that is not locally compact. Let $f \colon X \ra Z$ be a completed covering and $z \in Z.$ We say that $f$ has \textit{uniformly bounded local multiplicities} in $f^{-1}\{z\} \subset X$ if there exist a neighbourhood $U \in \mathcal{N}_Z(z)$ so that every point 
$z' \in U$ satisfies  
$$\sup\{\#(f^{-1} \{z'\} \cap D) : D \text{ is a component of }f^{-1}(U)\}< \infty.$$

\begin{figure}[htb]
\includegraphics{kaavio.1}
\caption{}\label{qer}
\end{figure}

\begin{Apulause}\label{kotiin}
Suppose $f \colon X \ra Z$ is a completed covering and $(Y,p,q)$ is a locally compact monodromy representation of $f.$ Then $f$ has uniformly bounded local multiplicities in $f^{-1}\{z\} \subset X$ for every $z \in Z.$  
\end{Apulause}
\begin{proof} Let $z \in Z$ and $y \in q^{-1}\{z\}.$ Since $q \colon Y \ra Z$ is an open and discrete map, the map $q$ has locally finite multiplicity, since $Y$ is locally compact. Thus there exists such $U \in \mathcal{N}_Z(z)$ that $q \r D_y$ has finite multiplicity for the $y$-component $D_y$ of $q^{-1}(U).$ Let $z' \in U$ and $k_{z'}:=\#(q^{-1}\{z'\} \cap D_y).$ Since $q$ is an orbit map, %$q \r D$ has finite multiplicity for every component $D_y$ of $q^{-1}(U).$ Indeed,
$$\sup\{\#(q^{-1} \{z'\} \cap D) : D \text{ is a component of }q^{-1}(U)\}= k_{z'}.$$ Hence  
$$\sup\{\#(f^{-1} \{z'\} \cap E) : E \text{ is a component of }f^{-1}(U)\}\leq k_{z'} < \infty,$$
since $q=f \circ p.$ Thus $f$ has uniformly bounded local multiplicities in $f^{-1}\{z\}.$
\end{proof}

\begin{Not} Suppose $f \colon X \ra Z$ is a completed covering between manifolds so that there exists $z \in Z$ so that $f$ does not have uniformly bounded local multiplicities in $f^{-1}\{z\} \subset X.$ Then, as a consequence of Lemma \ref{kotiin}, $f$ is not an orbit map. 
\end{Not}

\begin{Prop}There exists a BLD-mapping from $\R^2$ to the $2$-sphere $S^2$ having a monodromy representation that is not locally compact.
\end{Prop}
\begin{proof}Let $f \colon \R^2 \ra S^2$ be a BLD-mapping for which $f(B_f) \subset S^2$ is a finite set, but for a point $z \in S^2$ the map $f$ does not have uniformly bounded local multiplicities in $f^{-1}\{z\} \subset \R^2$; see Figure \ref{qer}. 

Since $f$ is BLD-mapping and $f(B_f)\subset S^2$ is closed, $f$ is a completed covering by \cite{L}. Since $f(B_f) \subset S^2$ is a finite set, $f$ is a stabily completed covering. Thus it has a monodromy representation $(Y,p,q)$ by Theorem \ref{hh}. By Proposition \ref{kotiin}, $(Y,p,q)$ is not locally compact.  
\end{proof}

\subsection{Completed normal coverings with locally finite multiplicity}\label{brr}

Let $Z$ be a PL manifold and $f \colon X \ra Z$ a stabily completed covering. Then there are large subsets $Y' \subset Y$ and $Z' \subset Z$ so that $g:=p \r Y' \colon Y' \ra Z'$ is a covering and $Z$ is $(Z', \mathrm{Ker}(\sigma_g))$-stable. 

Fix $z_0 \in Z'.$ Let $z \in Z$ and $U \in \mathcal{N}_Z(z,z_0)$ be a $(Z', \mathrm{Ker}(\sigma_g))$-stable neighbourhood of $z.$ Let $\pi \colon \pi(Z',z_0) \ra \pi(Z',z_0)/\mathrm{Ker}(\sigma_g)$ be the canonical map. 
\begin{Def}\label{br} A group $H \subset  \pi(Z',z_0)/\mathrm{Ker}(\sigma_g)$ is a \textit{local monodromy group} of the completed covering $f \colon X \ra Z$ at $z$ if there exists a $(Z', \mathrm{Ker}(\sigma_g))$-stable neighbourhood $U \in \mathcal{N}_Z(z,z_0)$ of $z$ so that $H=\mathrm{Im}(\pi \circ (\iota_{Z' \cap U, Z'})_*)$ for the inclusion $\iota_{Z' \cap U,Z'}$.
\end{Def}
For a $(Z', \mathrm{Ker}(\sigma_g))$-stable neighbourhood $U \in \mathcal{N}_Z(z,z_0)$ we denote
\begin{equation*}
\text{Mono}_f(U;Z'):=\mathrm{Im}(\pi \circ (\iota_{Z' \cap U, Z'})_*).
\end{equation*}

\begin{Not}If $U, V \in \mathcal{N}_Z(z,z_0)$ are $(Z', \mathrm{Ker}(\sigma_g))$-stable neighbourhoods of $z \in Z,$ then 
$$\text{Mono}_f(V;Z')=\text{Mono}_f(U;Z')$$ if $V \subset U.$  
\end{Not}   

\begin{Apulause}\label{whynot}Let $Z$ be a PL manifold and $p \colon Y \ra Z$ a stabily completed normal covering. Let $Y' \subset Y$ and $Z' \subset Z$ be such large subsets that $g:=p \r Y' \colon Y' \ra Z'$ is a normal covering and $Z$ is $(Z', \mathrm{Ker}(\sigma_g))$-stable. Let $y \in Y,$ $z:=p(y)$ and $V$ be a $(Z', \mathrm{Ker}(\sigma_g)))$-stable neighbourhood of the point $z.$  Let $Z$ have a triangulation $T$ having a vertex $z_0 \in Z'$ and let $n \geq 2$ be such that $\mathrm{St}_{1/n}(z)\subset V.$ Let $D_y$ be the $y$-component of $p^{-1}(\mathrm{St}_{1/n}(z)).$ Then
\begin{itemize} 
\item[(a)] 
$\#(p^{-1}\{z'\}\cap D_y) = \#\mathrm{Mono}_p(V;Z')$ for $z' \in \mathrm{St}_{1/n}(z) \cap Z'$ and
\item[(b)]$\#(p^{-1}\{z'\}\cap D_y) \leq \#\mathrm{Mono}_p(V;Z')$ for $z' \in \mathrm{St}_{1/n}(z).$
\end{itemize}
\end{Apulause}

\begin{proof} We first prove claim (a). Let $z'$ and $z''$ be points in $\mathrm{St}_{1/n}(z) \cap Z'.$ Then $\#(p^{-1}\{z'\}\cap D_y)=\#(p^{-1}\{z''\}\cap D_y),$ since $\mathrm{St}_{1/n}(z) \cap Z'$ is connected and $D_y \cap Y'$ is a component of $g^{-1}(\mathrm{St}_{1/n}(z) \cap Z').$ Thus we only need to show that there is a point $z_1 \in \mathrm{St}_{1/n}(z) \cap Z'$ satisfying $\#(p^{-1}\{z_1\}\cap D_y)=\#\text{Mono}_p(V).$

By Lemma \ref{D_2}, there is an open set $U \subset V \cap Z'$ so that $U \cap \mathrm{St}_{1/n}(z)$ is path-connected and $V_n:=U \cup \mathrm{St}_{1/n}(z) \in \mathcal{N}_Z(z,z_0).$ Since $V_n\subset V,$ the set $V_n$ is a $(Z', \mathrm{Ker}(\sigma_g))$-stable neighbourhood of $z.$ Hence 
$$\#\mathrm{Im}(\pi \circ (\iota_{V_n \cap Z',Z'})_*)=\#\mathrm{Im}(\pi \circ (\iota_{V \cap Z',Z'})_*).$$

Let $z_1 \in V_n \cap Z'$ and let $\be \colon [0,1]\ra V_n \cap Z'$ be a path $\be \colon z_0 \cu z_1.$ Let $y_0 \in p^{-1}\{z_0\}$ and $y_1:=\ol{\be}_{y_0}(1).$ Let $D_{y_1}$ be the $y_1$-component of $f^{-1}(\mathrm{St}_{1/n}(z)).$ Since $p$ is an orbit map, there is by Lemma \ref{vindpark} a deck-transformation $\tau \in \T(p)$ satisfying $\tau(D_y)=D_{y_1}.$ Hence
$$\#(p^{-1}\{z_1\}\cap D_y)=\#(p^{-1}\{z_1\}\cap D_{y_1}).$$

Since $g$ is a normal covering, we have $\mathrm{Ker}(\sigma_g)=\mathrm{Im}(g_*).$ By Lemma \ref{D_1} there is for every loop $\gamma \colon (S^1,e_0)\ra (V_n \cap Z',z_0)$ a loop $\al \colon (S^1,e_0)\ra (\mathrm{St}_{1/n}(z) \cap Z',z_1)$ satisfying $[\gamma]=[\be\al\be^{\law}]$ in $V_n \cap Z'.$ Hence 
$$\#(p^{-1}\{z_1\}\cap D_{y_1})=\#\mathrm{Im}(\pi \circ (\iota_{V_n \cap Z',Z'})_*),$$
since $g^{-1}(Z')=Y'$ by Lemma \ref{biotin}. 
We conclude that
\begin{align*}\#(p^{-1}\{z_1\}\cap D_y)=&\#(p^{-1}\{z_1\}\cap D_{y_1})=\#\mathrm{Im}(\pi \circ (\iota_{V_n \cap Z',Z'})_*)\\=&\#\mathrm{Im}(\pi \circ (\iota_{V \cap Z',Z'})_*)=\#\text{Mono}_p(V).
\end{align*}

We then prove claim (b). Since $p^{-1}\{z'\}$ is a countable set for every $z' \in Z,$ the statements holds trivially if $\#\text{Mono}_p(V)=\infty.$ Suppose $\#\text{Mono}_p(V)=k$ for $k \in \N.$  

Towards a contradiction suppose that $z'\in \mathrm{St}_{1/n}(z) \setminus Z'$ and there exists $k+1$ points $y_0, \ldots, y_k$ in $p^{-1}\{z'\}\cap D_y.$ 
Let $U\subset \mathrm{St}_{1/n}(z)$ be an open connected set so that the $y_i$-components $D'_{y_i}\subset D$ of $p^{-1}(U)$ are pairwise disjoint. By Lemma \ref{brunate}, $p(D'_{y_i})=U$ for every $i \in \{0, \ldots, k\}.$ Since $Z'\subset Z$ is large, there exists a point $z'_1 \in U \cap Z'$ for which $\#(p^{-1}\{z'\}\cap D_y)\geq k.$ This is a contradiction with (a). Hence $\#(p^{-1}\{z'\}\cap D_y)\leq k$ for all $z' \in \mathrm{St}_{1/n}(z).$
\end{proof}

\begin{Thm}\label{torstai} Let $Z$ be a PL manifold and $p \colon Y \ra Z$ a discrete completed normal covering. Then $p$ has locally finite multiplicity if and only if $p$ has a finite local monodromy group at each $z \in Z$. 
\end{Thm}

\begin{proof}The map $p$ is an orbit map by Theorem \ref{vauva}. Thus there is by Theorem \ref{TFAE} large subsets $Y' \subset X$ and $Z' \subset Z$ so that $g:=p \r Y' \colon Y' \ra Z'$ is a normal covering and $Z$ is $(Z', \mathrm{Ker}(\sigma_g))$-stable. Fix a triangulation $T$ of $Z$ having a vertex $z_0 \in Z'$. 

Suppose first that $p$ has locally finite multiplicity. Let $z \in Z$ and $y \in p^{-1}\{z\}.$ Let $D_n$ be the $y$-component of $p^{-1}(\mathrm{St}_{1/n}(z))$ for every $n \geq 1.$ Then $(D_n)_{n \in \N}$ is a neighbourhood basis of $y,$ since $p$ is a spread.

Let $V \in \mathcal{N}_Z(z,z_0)$ be a $(Z', \mathrm{Ker}(\sigma_g))$-stable neighbourhood of $z.$ Since $p$ has locally finite multiplicity, there exists $n \geq 2$ so that $\mathrm{St}_{1/n}(z)\subset V$ and $\#(p^{-1}\{z'\}\cap D_{n})<\infty$ for every $z' \in \mathrm{St}_{1/n}(z).$ Let $z_1 \in \mathrm{St}_{1/n}(z) \cap Z'.$ By Lemma \ref{whynot}, $$\#\text{Mono}_p(V)=\#(p^{-1}\{z_1\}\cap D_n)<\infty.$$ 

Suppose then that $p$ has a finite local monodromy group at each $z \in Z$. Let $y \in Y$ and $z:=p(y).$ Let $V\in \mathcal{N}_Z(z,z_0)$ be a $(Z', \mathrm{Ker}(\sigma_g))$-stable neighbourhood of the point $z$ so that $\#\text{Mono}_p(V)< \infty.$ Let $n \geq 2$ be such that $\mathrm{St}_{1/n}(z)\subset V$ and $D_y$ be the $y$-component of $p^{-1}(\mathrm{St}_{1/n}(z)).$ By Lemma \ref{whynot}, $p \r D_y$ has finite multiplicity. Thus $p$ has locally finite multiplicity. 
\end{proof}

Next we show regularity results for local monodromy.

\begin{Thm}\label{viimeista} Let $f \colon X \ra Z$ be a stabily completed covering between PL manifolds. 
Let 
\begin{equation*}
\xymatrix{
& Y \ar[ld]_{p} \ar[rd]^{q} &\\
X \ar[rr]^{f} & & Z }
\end{equation*}
be a commutative diagram of completed coverings so that $p$ and $q$ are orbit maps and $q$ is natural of $f.$ Then $q$ has locally finite multiplicity if and only if every $z \in Z$ has a finite local monodromy group of $f$ at $z.$ 
\end{Thm}
\begin{proof} By naturality of $q$ we may fix large subsets $X' \subset X,$ $Y' \subset Y$ and $Z' \subset Z$ for which $g:=f \r X' \colon X' \ra Z'$ and $h:=q \r Y' \colon Y' \ra Z'$ are coverings satisfying
$\mathrm{Ker}(\sigma_{g})=\mathrm{Ker}(\sigma_{h}).$

Suppose first that $f$ has a finite local monodromy group of $f$ at each point of $Z$. Let $z \in Z.$ Then there is a $(Z', \mathrm{Ker}(\sigma_g))$-stable neighbourhood $V \in \mathcal{N}_{Z}(z,z_0)$ of $z$ for which $\#\text{Mono}(V) <\infty.$ Since $\mathrm{Ker}(\sigma_{g})=\mathrm{Ker}(\sigma_{h}),$ $V$ is a $(Z', \mathrm{Ker}(\sigma_h))$-stable neighbourhood of $z.$ Since $\#\text{Mono}(V) <\infty,$ $q$ has a finite local monodromy group   at $z.$ Thus $q$ has locally finite multiplicity by Theorem \ref{torstai}. 

Suppose then that $q$ has locally finite multiplicity. Let $z \in Z.$ By Theorem \ref{torstai} $q$ has a finite local monodromy group at $z$. Since $\mathrm{Ker}(\sigma_{g})=\mathrm{Ker}(\sigma_{h}),$ the finite local monodromy group of $q$ at $z$ is a finite local monodromy group of $f$ at $z.$ 
\end{proof}

As a direct consequence of Theorem \ref{viimeista} we obtain an analogy of Theorem \ref{relas}.

\begin{Cor}\label{matka} Let $f \colon X \ra Z$ be a stabily completed covering between PL manifolds, so that $f$ has a finite local monodromy group at each $z \in Z.$ Then there is a monodromy representation $(Y,p,q)$ of $f,$ where $q$ has locally finite multiplicity. 
\end{Cor}

We conclude this section by defining for a PL manifold $Z,$ completed normal covering $p \colon Y \ra Z$ and $y \in Y,$ the homotopical index $\H(y,p).$ By Lemmas \ref{whynot} and \ref{torstai} we may define as follows.

\begin{Def}\label{prev}Let $Z$ be a PL manifold, $p \colon Y \ra Z$ a completed normal covering having locally finite multiplicity and $y \in Y.$ Let $Y' \subset Y$ and $Z' \subset Z$ be large subsets so that $g:=p \r Y' \colon Y' \ra Z'$ is a normal covering. Let $T$ be a triangulation of $Z$ having a vertex $z_0 \in Z',$ $n\geq 2$ be such that $\mathrm{St}_{1/n}(p(y))$ is contained in a $(Z', \mathrm{Ker}(\sigma_g))$-stable neighbourhood of $p(y),$ $D$ be the $y$-component of $\mathrm{St}_{1/n}(p(y))$ and $z' \in \mathrm{St}_{1/n}(p(y)) \cap Z'.$ 
Then $$\H(p,y):=\#(p^{-1}\{z'\} \cap D)$$ is the \textit{homotopical index} of $y$ in $p.$
\end{Def}

When $p \colon Y \ra Z,$ $y \in Y$ and $D \subset Y$ are as in definition \ref{prev} we get the following upper bound for the homotopical indices of points in $D.$ 

\begin{Apulause}\label{maxindeksi} Let $y' \in D$ and $\H(y',p)$ be the  homotopical index of $y'$ in $p.$ Then $\H(y',p) \leq \H(y,p).$ 
\end{Apulause}
\begin{proof} Let $y' \in D.$ Let $m \geq 2$ be such that $\mathrm{St}_{1/m}(p(y')) \subset \mathrm{St}_{1/n}(p(y))$ and such that $\mathrm{St}_{1/m}(p(y'))$ is contained in a $(Z', \mathrm{Ker}(\sigma_g))$-stable neighbourhood of $p(y').$ Let $C \subset D$ be the $y'$-component of $\mathrm{St}_{1/m}(p(y))$ and $z' \in \mathrm{St}_{1/m}(p(y')) \cap \mathrm{St}_{1/n}(p(y)) \cap Z'.$ Then by Lemma \ref{whynot} 
$$\H(p,y)=\#(p^{-1}\{z'\} \cap D) \leq \#(p^{-1}\{z'\} \cap C)=\H(p,y'),$$
since $C \subset D.$ 
\end{proof}

\subsection{Metrization of monodromy representations}\label{ccamm}

In this section we consider PL manifolds as length manifolds and prove Theorem \ref{kk} and Theorem \ref{hopo} in the introduction. 

Let $|K|\subset \R^n$ be the polyhedron of a simplical complex $K$ and
\begin{equation*}\label{tad}\ell(\gamma):=\sup\left\{\sum_{i=1}^{k-1}|\gamma(t_i)-\gamma(t_{i+1})| : t_1, \ldots, t_k \in [0,1], t_1 < \cdots < t_k, k \in \N \right\}
\end{equation*}
the \textit{length} of $\gamma$ for every path $\gamma$ in $|K|.$ 
A path $\gamma$ in $|K|$ is called \textit{rectifiable} if the length $\ell(\gamma)$ is finite. Since $|K| \subset \R^n$ is a polyhedron for $z_1, z_2 \in |K|$ there exists a rectifiable path $\gamma \colon z_1 \cu z_2.$ In particular, the formula $$d_s(z_1,z_2)=\text{inf}_{\gamma} \{\ell(\gamma) \mid \gamma \colon z_1 \cu z_2\}
$$
defines a path-metric on $|K|.$ For a detailed study of path length structures see \cite[Sec.\;1]{G}. 

We note that the path-metric $d_s$ coincides with the metric of $\R^n$ when restricted to any simplex $\sigma \in K.$ By local finiteness of $K$ the metric topology induced on $|K|$ by $d_s$ coincides with the relative topology of $|K|$ as a subset of $\R^n.$ 

For a PL manifold $Z,$ there exists an embedding $\iota \colon Z \hra \R^n$ satisfying $\iota(Z)=|K|$ for a simplical complex $K$ and for the next theorem we assume $Z=(|K|,d_s).$ For the next theorem we also assume the following, if $p \colon Y \ra Z$ is a completed normal covering and there are fixed large subsets $Y'\subset Y$ and $Z'\subset Z$ so that $p \r Y' \colon Y' \ra Z'$ is a normal covering, then $K$ has a vertex $z_0 \in Z'.$    
 
\begin{Thm}\label{kumpula}
Let $Z$ be a PL manifold, $p \colon Y \ra Z$ a completed normal covering that has locally finite multiplicity. Then there exists a path metric $d_s^*$ on $Y$ so that
\begin{itemize}
\item[(a)]the topology induced by $d_s^*$ on $Y$ coincide with the original topology of $Y,$ 
\item[(b)]$(Y,d_s^*)$ is a locally proper metric space,
\item[(c)]$p \colon (Y,d_s^*) \ra (Z,d_s)$ is a $1$-Lipschitz map and 
\item[(d)]every deck-transformation $\tau \in \T(p)$ is an isometry with respect to $d_s^*.$
\end{itemize}
\end{Thm}
The proof consists of Proposition \ref{pitkäperjantai} and Proposition \ref{lankalauantai}. We obtain metric $d_s^*$ in Theorem \ref{kumpula} by lifting paths. We show an analogy of the result \cite[II.3.4]{R} concerning local path-lifting for open discrete maps between manifolds. 

In Lemmas \ref{kiire}, \ref{kiireet} and \ref{indeksit} we assume that $Z$ is a PL manifold and $p \colon Y \ra Z$ is a completed normal covering that has locally finite multiplicity. Since $p$ has locally finite multiplicity, it is discrete by \cite[Thm.\;9.14]{MA}. We recall also that $p$ is an orbit map by Theorem \ref{vauva}, and that there are by Theorem \ref{TFAE} large subsets $Y' \subset Y$ and $Z' \subset Z$ so that $g:=p \r Y' \colon Y' \ra Z'$ is a normal covering and $Z$ is $(Z', \mathrm{Ker}(\sigma_g))$-stable. 

\begin{Apulause}\label{kiire} Let $z \in Z$ and $y \in p^{-1}\{z\}.$ Let $m \geq 2$ be such that $\mathrm{St}_{1/m}(z)$ is contained in a $(Z', \mathrm{Ker}(\sigma_g))$-stable neighbourhood of $z$ and let $D$ be the $y$-component of $p^{-1}(\mathrm{St}_{1/m}(z)).$ Let $\gamma \colon [0,1] \ra \mathrm{St}_{1/m}(z)$ be a path and let $\gamma' \colon [0,1] \ra D$ a function satisfying $p \circ \gamma'=\gamma.$ Then $\gamma'$ is continuous at every $t \in [0,1]$ satisfying $\H(\gamma'(t);p)=\H(y;p).$
\end{Apulause}
\begin{proof} Let $\mathrm{St}_{1/n}(\gamma(t))$ be the $1/n$-star of $\gamma(t)$ and $C_n$ the $\gamma'(t)$ component of $p^{-1}(\mathrm{St}_{1/n}(\gamma(t)))$ for every $n \geq 2.$ Let $n_0 \geq 2$ be such that $\mathrm{St}_{1/{n_0}}(\gamma(t))$ is contained in $\mathrm{St}_{1/m}(z)$ and in a $(Z', \mathrm{Ker}(\sigma_p))$-stable neighbourhood of $\gamma(t).$ Since $\H(\gamma'(t);p)=\H(y;p),$ $p^{-1}(\mathrm{St}_{1/n}(\gamma(t)))\cap D=C_n$ for every $n\geq n_0$ by Lemma \ref{whynot}.  Since $\gamma$ is con\-tinuous, there exists for every $n\geq n_0$ such $\epsilon_n > 0$ that $\gamma(t-\epsilon_n, t +\epsilon_n)\subset \mathrm{St}_{1/n}(\gamma(t)).$ Now for every $n\geq n_0$ we have $\gamma'(t- \epsilon_n, t +\epsilon_n)\subset C_n.$ Since $(C_n)_{n\geq n_0}$ is a neighbourhood basis at $\gamma'(t)$ the function $\gamma'$ is continuous at $t.$ 
\end{proof}

\begin{Apulause}\label{kiireet} Let $h \colon (0,1) \ra Z$ be a continuous map. Suppose there exists for every $t \in (0,1)$ such $\epsilon>0$ that for every $y \in p^{-1}\{h(t)\}$ there exists a continuous map $h' \colon (t-\epsilon, t+ \epsilon) \ra Y$ satisfying $h'(t)=y$ and $p \circ h' =h \r (t-\epsilon, t+ \epsilon).$ Then there exists for every $y \in p^{-1}\{h(1/2)\}$ a continuous map $\ol{h} \colon (0,1) \ra Y$ satisfying $\ol{h}(1/2)=y$ and $p \circ \ol{h}=h.$
\end{Apulause}
\begin{proof} Let $y \in p^{-1}\{h(1/2)\}.$ By Zorn's Lemma there exists a maximal connected set $I\subset (0,1), 1/2 \in I,$ for which there exists a continuous map $h' \colon I \ra Y$ satisfying $p \circ h'(1/2)=y$ and $p \circ h'=h \r I.$ By the existence of local lifts $I$ is an open interval $(a,b) \subset [0,1].$ We need to show that $a=0$ and $b=1.$ 

Fix a map $h' \colon I \ra Y$ satisfying $h'(1/2)=y$ and $p \circ h'=h \r I.$ Suppose $b\neq 1.$ Let $m \geq 2$ be such that $\mathrm{St}_{1/m}(h(b))$ is contained in a $(Z', \mathrm{Ker}(\sigma_g))$-stable neighbourhood of $h(b),$ $\delta \in (0,(b-a)/2)$ be such that $h(b -2\delta ,b) \subset \mathrm{St}_{1/m}(h(b)),$ $D$ be the component of $p^{-1}(\mathrm{St}_{1/m}(h(b)))$ that contains $h'(b -2\delta ,b),$ $y'$ be the point in $D \cap \{h(b)\}$ and $h''\colon (a,b] \ra Y$ be the extension of $h'$ defined by $h''(b)=y.$ 

By Lemma \ref{kiire}, $h'' \r [b-\delta,b] \colon [b-\delta,b] \ra D$ is continuous. Thus $h''$ is continuous. This is a contradiction with the maximality of $I=(a,b) \subset [0,1].$ Thus $b=1.$ By a similar argument $a=1.$    
\end{proof}

\begin{Apulause}\label{indeksit} Let $z \in Z$ and $m \geq 2$ such that $\mathrm{St}_{1/m}(z)$ is contained in a $(Z', \mathrm{Ker}(\sigma_g))$-stable neighbourhood of $z.$ Let $\gamma \colon [0,1] \ra \mathrm{St}_{1/m}(z)$ be a path satisfying $\gamma(0)=z$ and $y \in p^{-1}\{z\}.$ Then $\gamma$ has a lift $\ol{\gamma} \colon [0,1] \ra Y$ satisfying $\gamma(0)=y.$
\end{Apulause}
\begin{proof} 
Let $k:=\H(y,p).$ We prove the existence of the lift $\ol{\gamma}_y$ by induction on $k$. Let $D$ be the $y$-component of $p^{-1}(\mathrm{St}_{1/m}(z)).$ By Lemma \ref{brunate}, $p(D)=\mathrm{St}_{1/m}(z).$ Suppose $k=1.$ Then $p \r D \colon D \ra \mathrm{St}_{1/m}(z)$ is a homeomorphism. Thus $\gamma$ has a lift $\ol{\gamma}$ satisfying $\ol{\gamma}(0)=y.$ 

Suppose the statement holds for all $j \leq k-1.$ Let 
$$F:=\{z' \in \mathrm{St}_{1/m}(z) : \H(y',p)=\H(y,p) \text{ for every } y' \in p^{-1}\{z'\}\}$$ and
$$U:= \{z' \in \mathrm{St}_{1/m}(z) : \H(y',p)<\H(y,p) \text{ for every } y' \in p^{-1}\{z'\}\}.$$
By Lemma \ref{whynot} and Lemma \ref{maxindeksi}, $U=\mathrm{St}_{1/m}(z) \setminus F,$ $U \subset \mathrm{St}_{1/m}(z)$ is open and $p^{-1}\{z'\} \cap D$ is a point for every $z' \in F.$  

Now $\gamma^{-1}(U) \subset [0,1]$ is a countable union $\Cup_{n \in N}I_n$ of disjoint open intervals $I_n\subset [0,1].$ We use Lemma \ref{kiireet} to show that there exists for every $n \in N$ a continuous map $\ol{\gamma \r I_n} \colon I_n \ra D$ satisfying $p \circ \ol{\gamma \r I_n}= \gamma \r I_n.$ 

Let $n \in N$ and $t \in I_n.$ Then there exists such $p \geq 2$ that $\mathrm{St}_{1/p}(\gamma(t))$ is contained in a $(Z', \mathrm{Ker}(\sigma_g))$-stable neighbourhood of $\gamma(t)$ and such $\epsilon >0$ that $(t- \epsilon, t+\epsilon) \subset I_n$ and $\gamma(t- \epsilon, t+\epsilon) \subset \mathrm{St}_{1/p}(\gamma(t)).$ Let $h:=I_n \r (t- \epsilon, t + \epsilon).$ 

Let $y' \in p^{-1}(\gamma(t)).$ Since $\gamma(t) \in U,$ we have $\H(y',p) < k.$ Thus, by the induction hypotheses there are continuous maps $h_1' \colon (t- \epsilon ,t] \ra Y$ and $h_2' \colon [t,t + \epsilon) \ra Y$ satisfying $h_1'(t)=y'=h_2'(t),$ $p \circ h_1'=h \r (t- \epsilon,t]$ and $p \circ h_2'=h \r [t,t+ \epsilon).$ Hence there exists a continuous map $h' \colon (t - \epsilon, t + \epsilon) \ra Y$ satisfying $h'(t)=y'$ and $p \circ h'=h.$ 

Let $t_0 \in I_n$ be the center of $I_n$ and $y'' \in p^{-1}\{\gamma(t_0)\} \cap D.$ By Lemma \ref{kiireet}, there exists a continuous map $\ol{\gamma \r I_n} \colon I_n \ra Y$ satisfying $p \circ \ol{\gamma \r I_n}= \gamma \r I_n$ and $\ol{\gamma \r I_n}(t_0)=y'' \in D.$ Since $D$ is a component of $\mathrm{St}_{1/m}(z),$ $\ol{\gamma \r I_n} \colon I_n \ra D.$ 

Let then $J:=\gamma^{-1}(F).$  Since $p^{-1}\{z'\} \cap D$ is a point for every $z' \in F,$ there exists a unique function $\ol{\gamma \r J} \colon J \ra D$ satisfying $p \circ \ol{\gamma \r J}=\gamma \r J.$ 

Let $\ol{\gamma} \colon [0,1] \ra D$ be the unique function satisfying $\ol{\gamma} \r I_n=\ol{\gamma \r I_n}$ for every $n \in N$ and $\ol{\gamma} \r J=\ol{\gamma \r J}.$ Then $p \circ \ol{\gamma}=\gamma.$ Since $[0,1] \setminus J$ is open, the function $\ol{\gamma}$ is continuous at every $t \in [0,1] \setminus J.$ By Lemma \ref{kiire}, $\ol{\gamma}$ is continuous at every $t \in J.$ Thus $\ol{\gamma}$ is continuous. Since $\gamma(0)=z \in F,$ $\ol{\gamma}(0)=y.$
\end{proof}

\begin{Apulause}\label{LBAPp}Let $Z$ be a PL manifold and $p \colon Y \ra Z$ a completed normal covering that has locally finite multiplicity. Then for every pair of points $y_1$ and $y_2$ in $Y$ there exists a path $\gamma \colon y_1 \cu y_2$ so that $p \circ \gamma$ is rectifiable.
\end{Apulause}

\begin{proof} Let $Y' \subset Y$ and $Z' \subset Z$ be large subsets so that $g:=p \r Y' \colon Y' \ra Z'$ is a normal covering. 
By Lemma \ref{indeksit}, there exists points $z_1, z_2 \in Z'$ and rectifiable paths $\al_1 \colon p(y_1) \cu z_1$ and $\al_2 \colon p(y_2) \cu z_2$ having lifts $(\ol{\al_1})_{y_1}$ and   $(\ol{\al_2})_{y_2}$ by $p.$ We denote $y'_1=(\ol{\al_i})_{y_1}(1)$ and $y'_2=(\ol{\al_i})_{y_2}(1)$.
 
Since $g \colon Y' \ra Z'$ is a covering between open manifolds, there exists a path $\be \colon y_1' \cu y_2'$ so that $\ell(p \circ \be)<\infty.$ Now $\gamma:=(\ol{\al_1})_{y_1} \be (\ol{\al_2})_{y_1}^{\law} \colon y_1 \cu y_2$ satisfies
$$\ell(p \circ \gamma)\leq\ell(\al_1)+\ell(p \circ \be) + \ell(\al_2^{\law}) < \infty.$$ 
\end{proof}

Let $Z$ be a PL manifold and $p \colon Y \ra Z$ a completed normal covering that has locally finite multiplicity and $d_s$ the path-metric of $Z.$ We call $d_s^* \colon Y \times Y \ra \R_+$ defined by
$$d_s^*(y_1,y_2) = \inf \{\ell(p \circ \gamma) \mid p \circ \gamma,\, \gamma \colon y_1 \cu y_2\}$$
the \textit{pullback} of the path-metric $d_s$ by $p.$ 

\begin{Cor}\label{pitkäperjantai}Let $Z$ be a PL manifold and $p \colon Y \ra Z$ a completed normal covering that has locally finite multiplicity, $d_s$ the path-metric of $Z$ and $d_s^*$ the pullback of $d_s$ by $p.$ Then $d_s^*$ is a metric, every $\tau \in \T(p)$ is an isometry with respect to $d_s^*$ and $p \colon (Y,d_s^*) \ra (Y,d_s)$ is a $1$-Lipschitz map.
\end{Cor}
\begin{proof} The map $p$ is discrete by \cite[Thm.\;9.14]{MA}. Thus $d_s^*$ separates points of $Y$ and $d_s^*$ is a metric by Lemma \ref{LBAPp}; see \cite[Sec.\;1]{G}. By the definition of $d_s^*$ every $\tau \in \T(p)$ is an isometry with respect to $d_s^*$ and $p \colon (Y,d_s^*) \ra (Y,d_s)$ is a $1$-Lipschitz map.
\end{proof}

In the following proposition we show that the topology induced by $d_s^*$ on $Y$ coincides with the original topology of $Y.$ 

\begin{Prop}\label{kiirastorstai}Let $Z$ be a PL manifold and $p \colon Y \ra Z$ a completed normal covering that has locally finite multiplicity. Let $d_s$ be the path-metric on $Z$ and $\T$ the topology of $Y$. Let $d_s^*$ be the pullback of the path-metric $d_s$ of $Z$ by $p$ and $\T_{d_s^*}$ the topology induced on $Y$ by the metric $d_s^*.$ Then $\T=\T_{d_s^*}.$ 
\end{Prop}

\begin{proof} We first show that $\T \subset \T_{d_s^*}.$ Since $p$ is a spread, it is sufficient to show that for every open connected subset $V \subset Z$ and component $U$ of $p^{-1}(V)$ we have $U \in \T_{d_s^*}.$ Let $V \subset Z$ be an open and connected set, and let $U \subset Y$ be a component of $p^{-1}(V).$ Fix $y \in U.$ Since $p \colon (Y,d_s^*) \ra (Y,d_s)$ is a $1$-Lipschitz map, there exists $r_y \in (0,1)$ so that $p(B(y,r_y))\subset B(y,r_y) \subset V.$ Now for every point $y' \in B(y,r_y)$ there exists by the definition of $d_s^*$ a path $\gamma \colon y \cu y'$ in $p^{-1}(B(y,r_y))\subset p^{-1}(V).$
Thus $B(y,r_y)\subset U,$ since $U$ is the $y$-component of $p^{-1}(V).$ We conclude
$U \in \T_{d_s^*}.$

We then show that $\T_{d_s^*} \subset \T.$ Let $U \in \T_{d_s^*}.$ The map $p$ is discrete by \cite[Thm.\;9.14]{MA}, uniformly discrete Theorem \ref{TFAE} and an orbit map by Theorem \ref{vauva}. Thus, as a consequence of Lemma \ref{indeksit}, there exists for every $y \in U$ a radius $r_y \in (0,1)$ that satisfies the following conditions:
\begin{itemize}
\item[(a)]for the $y$-component $U_y$ of $p^{-1}(B(p(y),r_y)),$ $U_y \cap p^{-1}\{p(y)\}=\{y\},$  
\item[(b)] every path in $B(p(y),r_y)$ beginning at $p(y)$ has a total lift into $Y$ beginning at $y,$ 
\item[(c)] $[z,p(y)] \subset B(p(y),r_y)$ for every $z \in B(p(y),r_y)$ and 
\item[(d)] $B(y,r_y) \subset U.$ 
\end{itemize}
Since $p$ is a spread, $U_y \in \T$ for every $y \in U.$ 
It is suffices to show that $U_y \subset B(y,r_y)$ for every $y \in U,$ since then 
$$U=\Cup_{y \in U}B(y,r_y)=\Cup_{y \in U}U_y \in \T.$$

Let $y_1 \in U_y.$ Then $p(y_1) \in B(p(y),r_y)$ and there exists a path $\gamma \colon p(y) \cu p(y_1)$ in $B(p(y),r_y)$ satisfying $\ell(\gamma)<r_y.$

Let $\ol{\gamma}_y$ be a lift of $\gamma$ in $U_y$ beginning at $y.$ Then 
$$d_s^*(y, \ol{\gamma}_y(1)) \leq \ell(p \circ \ol{\gamma}_y)=\ell(\gamma)<r_y.$$
Hence $\ol{\gamma}_y(1) \in B(y,r_y).$

Since $p$ is an orbit map there is a deck-transformation $\tau \in \T(p)$ satisfying $\tau (\ol{\gamma}(1))=y_1.$ Since $y_1$ and $\ol{\gamma}_y(1)$ belong to $U_y,$ we have $\tau(U_y)=U_y$ by Lemma \ref{brunate}. Hence $\tau(y)=y,$ since $U_y \cap p^{-1}\{p(y)\}=\{y\}.$ Thus 
$$d_s^*(y_1,y)=d_s^*(\tau(\ol{\gamma}_y(1)), \tau(y))=d_s^*(\ol{\gamma}_y(1),y)<r_y,$$
since $\tau$ is an isometry with respect to $d_s^*.$  
Thus $y_1 \in B(y,r_y).$ Thus $U_y \subset B(y,r_y)$ and we conclude that $\T_{d_s^*} \subset \T.$ 
\end{proof}

In the following proposition we show that the topology induced by the pullback metric is locally proper.

\begin{Prop}\label{lankalauantai}Let $p \colon Y \ra Z$ be a completed normal covering onto a PL manifold $Z$ that has locally finite multiplicity. Let $d_s^*$ be the pullback of the path-metric $d_s$ of $Z$ by $p.$ Then $Y$ is a locally proper metric space with respect to $d_s^*$. 
\end{Prop}
\begin{proof} Fix $y \in Y.$ Since $p$ has locally finite multiplicity, Lemma \ref{indeksit} implies that there exists $r \in (0,1)$ that satisfies $p(\overline{B(y,r)})=\overline{B(p(y),r)}$ and satisfies for the $y$-component $D$ of $p^{-1}(B(p(y),2r))$ the following conditions:
\begin{itemize} 
\item[(a)]$p^{-1}\{p(y)\} \cap D=\{y\}$ 
\item[(b)]$p\r D$ has finite multiplicity and 
\item[(c)]$\overline{B(y,r)} \subset D.$ 
\end{itemize}
We prove the claim by showing that $\overline{B(y,r)}$ is compact. 

Let $\mathcal{U}$ be an open cover of $\overline{B(y,r)}.$ Since $p$ is an open map and the set $p^{-1}\{z\} \cap \overline{B(y,r)}$ is finite, there exists for every $z \in p(\overline{B(y,r)})$ a radius $r_z \in (0,1)$ so that for every $y' \in p^{-1}\{z\} \cap \overline{B(y,r)}$ there exists $U \in \mathcal{U}$ satisfying $B(y',r_z)\subset U$ and  $p(B(y',r_z))=B(z,r_z).$ 

Since $p(\overline{B(y,r)})=\overline{B(p(y),r)} \subset Z$ is compact, we may fix $z_1, \ldots, z_k \in p(\overline{B(y,r)})$ so that $\{B(z_i,r_{z_i}) : 1\leq i \leq k\}$ is an open cover of $p(\overline{B(y,r)}).$ The set $\{B(y',r_{z_i}) \mid y' \in p^{-1}\{z_i\} \cap \overline{B(y,r)}, i \in \{1, \ldots, k\}\}$ is now finite and for every $i \in \{1, \ldots, k\}$ and $y' \in p^{-1}\{z_i\} \cap \overline{B(y,r)}$ there exists $U \in \mathcal{U}$ so that $B(y',r_{z_i})\subset U.$ Thus it suffices to show that
$$\overline{B(y,r)} \subset V:=\Cup \{B(y',r_{z_i}) : y' \in p^{-1}\{z_i\} \cap \overline{B(y,r)}, i \in \{1, \ldots, k\}\}.$$

Let $y' \in \overline{B(y,r)}.$ Fix $i \in \{1, \ldots, k\}$ satisfying $p(y') \subset B(z_i,r_{z_i})$ and $e \in p^{-1}\{z_i\} \cap \overline{B(y,r)}.$ Then there exists a point $y'' \in  p^{-1}\{p(y')\} \cap B(e,r_{z_i}),$ since $p(B(e,r_{z_i}))=B(z_i,r_{z_i}).$ Let $\tau \colon (Y,d_s^*) \ra (Y, d_{s}^*)$ be a deck-transformation isometry satisfying $\tau(y'')=y'.$ Then $y' \in B(\tau(e),r_{z_i})$ and $\tau(e) \in p^{-1}\{z_i\}.$ Since $y''$ and $y'$ belong to $D,$ $\tau(D)=D$ by Lemma \ref{vindpark}. Thus $\tau(y)=y,$ since $p^{-1}\{p(y)\} \cap D=\{y\}.$ Hence 
$$d_s^*(\tau(e),y)=d_s^*(\tau(e),\tau(y))=d_s^*(e,y) \leq r.$$ Thus $\tau(e) \in p^{-1}\{z_i\} \cap \overline{B(y,r)}$ and $y' \in B(\tau(e),r_{z_i}) \subset V.$
We conclude that $\U$ has a finite subcover, $\overline{B(y,r)}\subset Y$ is compact and $(Y,d_s^*)$ is a locally proper metric space.  
\end{proof}

This concludes the proof of Theorem \ref{kumpula} and we are ready for the proofs of Theorems \ref{kuuskytkolme}, \ref{kk} and \ref{hopo} in the introduction. 

We say that a path-metric $d'_s$ on $Z$ is a \textit{polyhedral path-metric} on $Z,$ if there exists such a simplicial complex $K,$ polyhedron $|K|\subset \R^n$ and embedding $\iota \colon Z \ra \R^n$ satisfying $\iota(Z)=|K|,$ that $d'_s$ is the pullback of $d_s$ by $\iota$ for   the path-metric $d_s$ of $|K|.$

We recall that a completed normal covering $p \colon Y \ra Z$ is an open map and the domain $Y$ is by definition a Hausdorff space. In particular, $p$ has thus locally finite multiplicity, if $p$ is discrete and $Y$ is locally compact. 

\begin{proof}[Proof of Theorem \ref{kuuskytkolme}] Let $Z$ be a PL manifold and $p \colon Y \ra Z$ be a discrete completed normal covering. If $Y$ is locally compact, then $p$ has locally finite multiplicity. Suppose that $p$ has locally finite multiplicity. By Theorem \ref{kumpula}, there is a metric $d_s^*$ on $Y$ so that the topology induced by $d_s^*$ coincides with the original topology of $Y$ and $(Y,d_s^*)$ is a locally proper metric space. Thus $Y$ is locally compact.  
\end{proof}

\begin{proof}[Proof of Theorem \ref{kk}] Let $f \colon X \ra Z$ be a completed covering between PL manifolds, $(Y,p,q)$ a locally compact monodromy representation of $f,$ $d_s$ a polyhedral path-metric of $Z.$ By Theorem \ref{vauva}, the orbit map $q$ is discrete. Thus $q$ has locally finite multiplicity, since $Y$ is locally compact. Since $q$ has locally finite multiplicity, by Theorem \ref{kumpula}, the pullback $d_s^*$ of $d_s$ is a path-metric on $Y$ satisfying conditions (a)--(d) in Theorem \ref{vauva}. 
\end{proof} 

\begin{proof}[Proof of Theorem \ref{hopo}]For the statement we need to show that a completed covering  $f \colon X \ra Z$ between PL manifolds has a locally compact monodromy representation $(Y,p,q)$ if and only if $f$ is stabily completed and $f$ has a finite local monodromy group at each $z \in Z.$ 

Suppose first that $f$ is stabily completed covering and $f$ has a finite local mo\-nodromy group at each $z \in Z.$ By Corollary \ref{matka} $f$ has a monodromy representation $(Y,p,q),$ where $q$ has locally finite multiplicity. By Theorem \ref{kumpula}, there is a metric $d_s^*$ on $Y$ so that the topology induced by $d_s^*$ coincides with the original topology of $Y$ and $(Y,d_s^*)$ is a locally proper metric space. Thus $Y$ is locally compact and $(Y,p,q)$ a locally compact monodromy representation of $f.$

Suppose then that $f$ has a locally compact monodromy representation $(Y,p,q).$ Then the map $q$ is discrete by Theorem \ref{vauva}. Thus $q$ has locally finite multiplicity. By Theorem \ref{TFAE}, $q$ is a stabily completed normal covering. Thus, by Theorem \ref{torstai}, $q$ has a finite local monodromy group at each $z \in Z,$ since $q$ has locally finite multiplicity. 

Fix $z \in Z$ and let $H$ be a finite local monodromy group of $q$ at $z.$ By Theorem \ref{spread_3} and Remark \ref{stabhuomio} there is a quotient of $H$ that is a local monodromy group of $f$ at $z.$ Since every quotient of a finite group is finite, $f$ has a finite local monodromy group at $z.$
\end{proof}

\bibliography{viite}{}
\bibliographystyle{amsplain}
%\bibliography{mybib}{}
%\bibliographystyle{plain}

%\bibliographystyle{abbrv}
%

\end{document}